\documentclass[12pt, reqno]{amsart}

\usepackage{cancel}
\usepackage{stmaryrd}
\usepackage{amsmath}
\usepackage{amsfonts}
\usepackage{amssymb}
\usepackage[all,cmtip]{xy}           
\usepackage{xspace}
\usepackage{bbding}
\usepackage{txfonts}
\usepackage[shortlabels]{enumitem}
\usepackage{cite}

\usepackage{tikz}
\usepackage{titletoc}

\usepackage{ifpdf}
\ifpdf
 \usepackage[colorlinks,final,backref=page,hyperindex]{hyperref}
\else
 \usepackage[colorlinks,final,backref=page,hyperindex,hypertex]{hyperref}
\fi

\usepackage[active]{srcltx}

\makeatletter

\begin{document}
\newcommand {\emptycomment}[1]{} 

\baselineskip=15pt
\newcommand{\nc}{\newcommand}
\newcommand{\delete}[1]{}
\nc{\mfootnote}[1]{\footnote{#1}} 
\nc{\todo}[1]{\tred{To do:} #1}

\nc{\mlabel}[1]{\label{#1}}  
\nc{\mcite}[1]{\cite{#1}}  
\nc{\mref}[1]{\ref{#1}}  
\nc{\meqref}[1]{\eqref{#1}} 
\nc{\mbibitem}[1]{\bibitem{#1}} 
\delete{
\nc{\mlabel}[1]{\label{#1}  
{\hfill \hspace{1cm}{\bf{{\ }\hfill(#1)}}}}
\nc{\mcite}[1]{\cite{#1}{{\bf{{\ }(#1)}}}}  
\nc{\mref}[1]{\ref{#1}{{\bf{{\ }(#1)}}}}  
\nc{\meqref}[1]{\eqref{#1}{{\bf{{\ }(#1)}}}} 
\nc{\mbibitem}[1]{\bibitem[\bf #1]{#1}} 
}

\newcommand {\comment}[1]{{\marginpar{*}\scriptsize\textbf{Comments:} #1}}
\nc{\mrm}[1]{{\rm #1}}
\nc{\id}{\mrm{id}}  \nc{\Id}{\mrm{Id}}
\nc{\admset}{\{\pm x\}\cup (-x+K^{\times})}

\def\a{\alpha}
\def\admt{admissible to~}
\def\asi{ASI~}
\def\aybe{abhcYBe~}
\def\aybes{abhcYBes~}
\def\b{\beta}
\def\bd{\boxdot}
\def\bbf{\bar{f}}
\def\bF{\bar{F}}
\def\bbF{\bar{\bar{F}}}
\def\bbbf{\bar{\bar{f}}}
\def\bg{\bar{g}}
\def\bG{\bar{G}}
\def\bh{BiHom-}
\def\bbG{\bar{\bar{G}}}
\def\bbg{\bar{\bar{g}}}
\def\bT{\bar{T}}
\def\bt{\bar{t}}
\def\bbT{\bar{\bar{T}}}
\def\bbt{\bar{\bar{t}}}
\def\bR{\bar{R}}
\def\br{\bar{r}}
\def\bbR{\bar{\bar{R}}}
\def\bbr{\bar{\bar{r}}}
\def\bu{\bar{u}}
\def\bU{\bar{U}}
\def\bbU{\bar{\bar{U}}}
\def\bbu{\bar{\bar{u}}}
\def\bw{\bar{w}}
\def\bW{\bar{W}}
\def\bbW{\bar{\bar{W}}}
\def\bbw{\bar{\bar{w}}}
\def\btl{\blacktriangleright}
\def\btr{\blacktriangleleft}
\def\ASIB{\mathbb{ASIB}}
\def\DB{\mathbb{DB}}
\def\MP{\mathbb{MP}}
\def\ci{\circ}
\def\d{\delta}
\def\dd{\diamondsuit}
\def\D{\Delta}
\def\fb{\mathfrak{B}}
\def\frakB{\mathfrak{B}}
\def\G{\Gamma}
\def\g{\gamma}
\def\k{\kappa}
\def\l{\lambda}
\def\lh{\leftharpoonup}
\def\lr{\longrightarrow}
\def\mf{\mathfrak{F}}
\def\my{\mathfrak{Y}}
\def\mx{\mathfrak{X}}
\def\mz{\mathfrak{Z}}
\def\mba{monoidal BiHom-AE }
\def\mb{monoidal BiHom-}
\def\mo{$\mathcal{O}$-}
\def\mr{\mathfrak{R}}
\def\ms{\mathfrak{S}}
\def\N{Nijenhuis~}
\def\o{\otimes}
\def\om{\omega}
\def\opa{\cdot_{A}}
\def\opb{\cdot_{B}}
\def\p{\psi}
\def\sadm{$\my$-admissible~}
\def\r{\rho}
\def\ra{\rightarrow}
\def\rh{\rightharpoonup}
\def\rr{r^{\#}}
\def\s{\sigma}
\def\st{\star}
\def\ti{\times}
\def\tl{\triangleright}
\def\tr{\triangleleft}
\def\v{\varepsilon}
\def\vp{\varphi}
\def\vth{\vartheta}

\newtheorem{thm}{Theorem}[section]
\newtheorem{theorem}[thm]{Theorem}
\newtheorem{lemma}[thm]{Lemma}
\newtheorem{corollary}[thm]{Corollary}
\newtheorem{proposition}[thm]{Proposition}
\theoremstyle{definition}
\newtheorem{definition}[thm]{Definition}
\newtheorem{example}[thm]{Example}
\newtheorem{remark}[thm]{Remark}
\newtheorem{pdef}[thm]{Proposition-Definition}
\newtheorem{condition}[thm]{Condition}
\newtheorem{question}[thm]{Question}
\newtheorem{counterexample}[thm]{Counterexample}
\renewcommand{\labelenumi}{{\rm(\alph{enumi})}}
\renewcommand{\theenumi}{\alph{enumi}}

\nc{\w}[1]{\textcolor{red}{WDG:#1}}
\font\cyr=wncyr10


\title{Averaging Zinbiel bialgebras}

\author[Li]{Bei Li}
\address{School of Mathematical Sciences, Qufu Normal University,  Qufu 273165, P. R. of China}
\email{libei9717@163.com}

\author[Wang]{Dingguo Wang \textsuperscript{*}}
\address{School of Mathematical Sciences, Qufu Normal University,  Qufu 273165, P. R. of China}
\email{dgwang@qfnu.edu.cn}

 \thanks{\textsuperscript{*}Corresponding author}

\date{\today}
 \begin{abstract}
In this paper, we introduce and study averaging Zinbiel bialgebras. We first
define representations of averaging Zinbiel algebras and prove a semi-direct
product characterization. We then introduce matched pairs and Manin triples
of averaging Zinbiel algebras, and characterize averaging Zinbiel bialgebras
by these two structures. We show that every averaging Zinbiel algebra induces
a pre-Lie algebra, and that every averaging Zinbiel bialgebra induces a pre-Lie bialgebra. We also study coboundary averaging Zinbiel bialgebras and introduce the \(S\) admissible Zinbiel Yang-Baxter equation.
Its symmetric solutions give rise to averaging Zinbiel bialgebras. Finally, we
investigate averaging operators on quadratic Rota-Baxter Zinbiel algebras
and relative Rota-Baxter operators on averaging Zinbiel algebras.
\end{abstract}

\medskip
\subjclass[2020]{
17B38, 
16T10,   
16T25.   
}
\keywords{Averaging Zinbiel algebra; Averaging Zinbiel bialgebra; Averaging Zinbiel Yang-Baxter equation; Relative Rota-Baxter operator}

\maketitle

\tableofcontents

\section{Introduction}
Zinbiel algebras, also called dual Leibniz algebras, were introduced by
Loday in the study of Leibniz algebra cohomology \cite{LodayPirashvili1993}.
They are Koszul dual to Leibniz algebras in operad theory
\cite{GinzburgKapranov1994,LodayVallette2012}, and appear in shuffle type
products, rack cohomology, deformation theory and multiple zeta values
\cite{Chapoton2022,CovezFarinatiLebedManchon2023}. Recall that a Zinbiel
algebra is a vector space \(A\) equipped with a bilinear product \(\circ\)
satisfying
\[
x\circ (y\circ z)=(x\circ y+y\circ x)\circ z,
\qquad \forall x,y,z\in A.
\]
The symmetrized product \(x\cdot y=x\circ y+y\circ x\) is commutative and
associative. Thus Zinbiel algebras are naturally connected with
commutative associative algebras, dendriform algebras, pre-Lie algebras and
Leibniz type structures \cite{LodayVallette2012,WY2026}. There has been extensive research on Zinbiel algebras and related structures; see, for example, \cite{GaoGuoZhang2021,GaoYangZhang2026,ShengTangZhu2023}.

Bialgebra structures are used to describe the compatibility between algebra
and coalgebra structures. Lie bialgebras were introduced by Drinfel'd in
connection with Poisson Lie groups and quantum groups \cite{Drinfeld1983},
and they can be characterized by Manin triples and matched pairs of Lie
algebras \cite{KosmannSchwarzbach2004}. Similar bialgebra theories have
been developed for associative, infinitesimal, dendriform and pre-Lie
algebras \cite{Aguiar2001,Aguiar2004,Bai2008}. Recently, Zinbiel
bialgebras were introduced in \cite{WY2026}, where Zinbiel bialgebras,
matched pairs of Zinbiel algebras and Manin triples of Zinbiel algebras
were shown to be equivalent. Moreover, the coboundary case gives rise to
the Zinbiel Yang-Baxter equation, and factorizable Zinbiel bialgebras are
characterized by quadratic Rota-Baxter Zinbiel algebras of nonzero weight
\cite{WY2026}.

Averaging operators form another important class of operators in algebra.
They first appeared in Reynolds' work on turbulence theory
\cite{Reynolds1895}, and were later studied in associative, Lie, Leibniz,
pre-Lie and other operated algebraic settings
\cite{Aguiar2000,Cao2000,Guo2012}. For an algebra \((A,\circ)\), a linear
map \(P:A\to A\) is called an averaging operator if
\[
P(x)\circ P(y)=P(P(x)\circ y)=P(x\circ P(y)),
\qquad \forall x,y\in A.
\]
Averaging Lie bialgebras were studied in \cite{HouShengZhou2026}, where
matched pairs, Manin triples, quadratic structures, Rota-Baxter
constructions and Yang-Baxter constructions were developed. Averaging pre-Lie bialgebras were then considered in \cite{GaoYangZhang2026}.  Averaging (anti-symmetric)
infinitesimal bialgebras were also studied in \cite{HouCui2027,ZhaoLiu}. Motivated
by these works and by the bialgebra theory of Zinbiel algebras in
\cite{WY2026}, we study averaging operators in the Zinbiel bialgebra
setting.

We introduce the representations of averaging Zinbiel algebras and prove
a semi-direct product characterization. We also describe the coregular
representation, which is needed for matched pairs and bialgebras. Then we
introduce matched pairs and Manin triples of averaging Zinbiel algebras.
Given two averaging Zinbiel algebras, we give necessary and sufficient
conditions under which their direct sum carries an averaging Zinbiel
algebra structure. We also show that every averaging Zinbiel algebra
induces a pre-Lie algebra through
\[
x*_P y=P(x)\circ y-x\circ P(y),
\qquad \forall x,y\in A.
\]
Moreover, matched pairs of averaging Zinbiel algebras induce matched pairs
of pre-Lie algebras.

We then define averaging Zinbiel bialgebras. An averaging Zinbiel bialgebra
consists of a Zinbiel bialgebra \((A,\circ,\Delta)\), an averaging operator
\(P\) on the algebra part, an averaging operator \(S\) on the coalgebra part,
and suitable admissibility conditions connecting \(P\) and \(S\). We prove
that averaging Zinbiel bialgebras are characterized by coregular matched
pairs of averaging Zinbiel algebras and by Manin triples of averaging
Zinbiel algebras. We also show that every averaging Zinbiel bialgebra
induces a pre-Lie bialgebra.

Finally, we study two constructions of averaging Zinbiel bialgebras. The
first one comes from coboundary Zinbiel bialgebras. For \(r\in A\otimes A\),
we consider
\[
\Delta_r(x)=
\bigl(\operatorname{id}\otimes(L_x+R_x)-L_x\otimes\operatorname{id}\bigr)r,
\qquad \forall x\in A,
\]
and introduce the \(S\) admissible Zinbiel Yang-Baxter equation. Its
symmetric solutions give rise to averaging Zinbiel bialgebras. The second one comes
from quadratic Rota-Baxter Zinbiel algebras and relative Rota-Baxter
operators. Compatible averaging operators on quadratic Rota-Baxter Zinbiel
algebras produce averaging Zinbiel bialgebras, while relative Rota-Baxter
operators provide descendent averaging Zinbiel algebras, matched pairs and
solutions of the \(S\) admissible Zinbiel Yang-Baxter equation.

The paper is organized as follows. Section~\ref{se:Representations of averaging Zinbiel algebras} studies representations,
semi-direct products and coregular representations of averaging Zinbiel
algebras. Section~\ref{se:Matched pair and Manin triple of averaging Zinbiel algebras} is devoted to matched pairs and Manin triples of
averaging Zinbiel algebras. Section~\ref{se:Averaging Zinbiel bialgebras} introduces averaging Zinbiel
bialgebras, proves their matched pair and Manin triple characterizations,
and studies their relations with induced pre-Lie bialgebras and quadratic
Rota-Baxter Zinbiel algebras. Section~\ref{se:Admissible Zinbiel Yang-Baxter equations} studies the \(S\) admissible
Zinbiel Yang-Baxter equation and relative Rota-Baxter operators on
averaging Zinbiel algebras.

\noindent\textbf{Notation.} Throughout this paper, we fix a field $\mathbf{k}$ of characteristic zero, and all vector spaces and algebras are assumed to be finite-dimensional over $\mathbf{k}$.

\section{Representations of averaging Zinbiel algebras}\label{se:Representations of averaging Zinbiel algebras}
In this section, we introduce averaging Zinbiel algebras and their
representations. We first record the basic definitions and examples, and
then show that every averaging Zinbiel algebra induces a pre-Lie algebra.
After that, we define representations of averaging Zinbiel algebras and
give a semi-direct product characterization. Finally, we discuss dual
representations and the coregular representation, which will be used in the
matched pair and bialgebra constructions below.
\begin{definition}
Let $(A,\circ)$ be a Zinbiel algebra and let $P:A\to A$ be a linear map.
If
\[
P(x)\circ P(y)
=
P\bigl(P(x)\circ y\bigr)
=
P\bigl(x\circ P(y)\bigr),
\qquad \forall x,y\in A,
\]
then $P$ is called an {\bf averaging operator} on $(A,\circ)$, and the pair
$((A,\circ),P)$ is called an {\bf averaging Zinbiel algebra}.
\end{definition}

\begin{definition}\label{def:homomorphism-avg-Zinbiel}
Let $((A,\circ),P)$ and $((A',\circ'),P')$ be two averaging Zinbiel
algebras. A linear map $f:A\to A'$ is called a homomorphism from
$((A,\circ),P)$ to $((A',\circ'),P')$ if
\[
f(x\circ y)=f(x)\circ' f(y),
\qquad \forall x,y\in A,
\]
and
\[
f(P(x))=P'(f(x)),
\qquad \forall x\in A.
\]
Furthermore, if $f$ is invertible, then $f$ is called an {\bf isomorphism}.
\end{definition}

\begin{example}\label{ex:basic-avg-Zinbiel}
Let $(A,\circ)$ be a Zinbiel algebra.

\begin{itemize}
\item[\emph{(a)}]
The identity map $P=\operatorname{id}_A$ is an averaging operator on
$(A,\circ)$.

\item[\emph{(b)}]
If a linear map $P:A\to A$ commutes with the left and right multiplication,
namely, for all $x\in A$,
\[
PL_x=L_xP
\qquad\text{and}\qquad
PR_x=R_xP,
\]
where
\[
L_x(y)=x\circ y,
\qquad
R_x(y)=y\circ x,
\]
then $P$ is an averaging operator on $(A,\circ)$.
\end{itemize}
\end{example}

The following result shows that averaging operators on Zinbiel algebras
produce pre-Lie algebra structures in a natural way.
\begin{proposition}\label{prop:induced-prelie-from-avg-Zinbiel}
Let $((A,\circ),P)$ be an averaging Zinbiel algebra. Define a bilinear
operation $*:A\otimes A\to A$ by
\[
x*y:=P(x)\circ y-x\circ P(y),
\qquad \forall x,y\in A.
\]
Then $(A,*)$ is a pre-Lie algebra, that is,
\[
(x*y)*z-x*(y*z)=(y*x)*z-y*(x*z),
\qquad \forall x,y,z\in A.
\]
The pre-Lie algebra $(A,*)$ is called the {\bf induced pre-Lie algebra} of the
averaging Zinbiel algebra $((A,\circ),P)$.
\end{proposition}

\begin{proof}
For all $x,y,z\in A$, since $P$ is an averaging operator, we first have
\[
\begin{aligned}
P(x*y)
&=P\bigl(P(x)\circ y-x\circ P(y)\bigr)\\
&=P(P(x)\circ y)-P(x\circ P(y))\\
&=P(x)\circ P(y)-P(x)\circ P(y)\\
&=0.
\end{aligned}
\]
Therefore,
\[
\begin{aligned}
(x*y)*z-x*(y*z)
&=P(x*y)\circ z-(x*y)\circ P(z)   \\
&\quad -P(x)\circ (y*z)+x\circ P(y*z)\\
&=-(P(x)\circ y-x\circ P(y))\circ P(z)\\
&\quad -P(x)\circ\bigl(P(y)\circ z-y\circ P(z)\bigr)\\
&=-(P(x)\circ y)\circ P(z)+(x\circ P(y))\circ P(z)\\
&\quad -P(x)\circ(P(y)\circ z)+P(x)\circ(y\circ P(z)).
\end{aligned}
\]
Using the Zinbiel identity, we obtain
\[
P(x)\circ(P(y)\circ z)
=
\bigl(P(x)\circ P(y)+P(y)\circ P(x)\bigr)\circ z
\]
and
\[
P(x)\circ(y\circ P(z))
=
\bigl(P(x)\circ y+y\circ P(x)\bigr)\circ P(z).
\]
Hence
\[
\begin{aligned}
(x*y)*z-x*(y*z)
&=-(P(x)\circ y)\circ P(z)+(x\circ P(y))\circ P(z)\\
&\quad -\bigl(P(x)\circ P(y)+P(y)\circ P(x)\bigr)\circ z\\
&\quad +\bigl(P(x)\circ y+y\circ P(x)\bigr)\circ P(z)\\
&=(x\circ P(y))\circ P(z)+(y\circ P(x))\circ P(z)\\
&\quad -\bigl(P(x)\circ P(y)+P(y)\circ P(x)\bigr)\circ z.
\end{aligned}
\]
The last expression is symmetric in $x$ and $y$. Thus,
\[
(x*y)*z-x*(y*z)=(y*x)*z-y*(x*z),
\qquad \forall x,y,z\in A.
\]
This proves that $(A,*)$ is a pre-Lie algebra.
\end{proof}

We now turn to representations of averaging Zinbiel algebras. These
representations will be used to describe semi-direct products and matched
pairs.
\begin{definition}
Let $((A,\circ),P)$ be an averaging Zinbiel algebra. A {\bf representation} of
$((A,\circ),P)$ is a pair $((V;\rho,\mu),\alpha)$, where

\begin{itemize}
\item[\emph{(a)}]\cite{WY2026} $(V;\rho,\mu)$ is a representation of the Zinbiel algebra
$(A,\circ)$, namely $\rho,\mu:A\to \operatorname{End}(V)$ are linear maps
satisfying
\begin{eqnarray}
&\rho(x)\rho(y)
=
\rho(x\circ y)+\rho(y\circ x),&\label{eq:representation of Zinbiel algebra1}\\
&\rho(x)\mu(y)
=
\mu(x\circ y)
=
\mu(y)\rho(x)+\mu(y)\mu(x),&\label{eq:representation of Zinbiel algebra2}
\qquad \forall x,y\in A;
\end{eqnarray}

\item[\emph{(b)}] $\alpha:V\to V$ is a linear map;

\item[\emph{(c)}] for all $x\in A$ and $v\in V$, the following compatibility
conditions hold:
\begin{eqnarray}
&\rho(P(x))\alpha(v)
=
\alpha\bigl(\rho(P(x))v\bigr)
=
\alpha\bigl(\rho(x)\alpha(v)\bigr),&\label{eq:representation of averagingZinbiel algebra1}\\
&\mu(P(x))\alpha(v)
=
\alpha\bigl(\mu(P(x))v\bigr)
=
\alpha\bigl(\mu(x)\alpha(v)\bigr).&\label{eq:representation of averagingZinbiel algebra2}
\end{eqnarray}
\end{itemize}
\end{definition}

\begin{definition}
Let $((V_1;\rho_1,\mu_1),\alpha_1)$ and
$((V_2;\rho_2,\mu_2),\alpha_2)$ be two representations of an averaging
Zinbiel algebra $((A,\circ),P)$. A linear map $f:V_1\to V_2$ is called a
{\bf homomorphism} from $((V_1;\rho_1,\mu_1),\alpha_1)$ to
$((V_2;\rho_2,\mu_2),\alpha_2)$ if, for all $x\in A$ and $v\in V_1$,
\[
\rho_2(x)f(v)=f\bigl(\rho_1(x)v\bigr),
\qquad
\mu_2(x)f(v)=f\bigl(\mu_1(x)v\bigr),
\qquad
f(\alpha_1(v))=\alpha_2(f(v)).
\]
Furthermore, if $f$ is invertible, then $f$ is called an {\bf isomorphism}.
\end{definition}

\begin{example}
Let $((A,\circ),P)$ be an averaging Zinbiel algebra. Define
$L,R:A\to \operatorname{End}(A)$ by
\[
L_x(y)=x\circ y,
\qquad
R_x(y)=y\circ x,
\qquad \forall x,y\in A.
\]
Then $((A;L,R),P)$ is a representation of $((A,\circ),P)$, called the
{\bf regular representation} of $((A,\circ),P)$.
\end{example}

\begin{proposition}\label{prop:semi-direct-avg-Zinbiel}
Let $((A,\circ),P)$ be an averaging Zinbiel algebra, let
$(V;\rho,\mu)$ be a representation of the Zinbiel algebra $(A,\circ)$,
and let $\alpha:V\to V$ be a linear map. Define a product
$\circ_{\rho,\mu}$ on $A\oplus V$ by
\[
(x+u)\circ_{\rho,\mu}(y+v)
:=
x\circ y+\rho(x)v+\mu(y)u,
\qquad x,y\in A,\ u,v\in V,
\]
and define a linear map $P_{A\oplus V}:A\oplus V\to A\oplus V$ by
\[
P_{A\oplus V}(x+u):=(P+\alpha)(x+u)=P(x)+\alpha(u).
\]
Then $A\oplus V$, equipped with the product $\circ_{\rho,\mu}$ and the
linear map $P_{A\oplus V}$, is an averaging Zinbiel algebra if and only if
$((V;\rho,\mu),\alpha)$ is a representation of the averaging Zinbiel algebra
$((A,\circ),P)$.

In this case, we call
\[
(A\ltimes_{\rho,\mu}V,\ P+\alpha)
\]
a semi-direct product of averaging Zinbiel algebras.
\end{proposition}

\begin{proof}
First, by \cite{WY2026}, we know that
$(A\oplus V,\circ_{\rho,\mu})=:A\ltimes_{\rho,\mu}V$ is a Zinbiel algebra
if and only if $(V;\rho,\mu)$ is a representation of $(A,\circ)$.

Second, by a direct calculation, for all $x,y\in A$ and $u,v\in V$, we have
\[
\begin{aligned}
P_{A\oplus V}(x+u)\circ_{\rho,\mu}P_{A\oplus V}(y+v)
&=
(P(x)+\alpha(u))\circ_{\rho,\mu}(P(y)+\alpha(v))  \\
&=
P(x)\circ P(y)+\rho(P(x))\alpha(v)+\mu(P(y))\alpha(u).
\end{aligned}
\]
Similarly, we can get
\[
\begin{aligned}
P_{A\oplus V}\bigl(P_{A\oplus V}(x+u)\circ_{\rho,\mu}(y+v)\bigr)
&=
P(P(x)\circ y)+\alpha(\rho(P(x))v)+\alpha(\mu(y)\alpha(u)).
\end{aligned}
\]

\[
\begin{aligned}
P_{A\oplus V}\bigl((x+u)\circ_{\rho,\mu}P_{A\oplus V}(y+v)\bigr)
&=
P(x\circ P(y))+\alpha(\rho(x)\alpha(v))+\alpha(\mu(P(y))u).
\end{aligned}
\]

Since \(P\) is an averaging operator on \((A,\circ)\), $P_{A\oplus V}$ is an averaging operator on
$A\ltimes_{\rho,\mu}V$ if and only if
\[
\rho(P(x))\alpha(v)
=
\alpha(\rho(P(x))v)
=
\alpha(\rho(x)\alpha(v)),
\]
and
\[
\mu(P(x))\alpha(v)
=
\alpha(\mu(P(x))v)
=
\alpha(\mu(x)\alpha(v)),
\]
for all $x\in A$ and $v\in V$. Hence the result follows.
\end{proof}

\begin{proposition}\label{prop:dual-rep-avg-Zinbiel}
Let $((A,\circ),P)$ be an averaging Zinbiel algebra, let
$(V;\rho,\mu)$ be a representation of the Zinbiel algebra $(A,\circ)$, and let
$\beta:V\to V$ be a linear map. Define linear maps
$\rho^*,\mu^*:A\to \operatorname{End}(V^*)$ by
\[
\langle \rho^*(x)v^*,v\rangle=-\langle v^*,\rho(x)v\rangle,
\qquad
\langle \mu^*(x)v^*,v\rangle=-\langle v^*,\mu(x)v\rangle,
\]
for all $x\in A$, $v^*\in V^*$ and $v\in V$. Then
\[
\bigl((V^*;-\rho^*-\mu^*,\mu^*),\beta^*\bigr)
\]
is a representation of the averaging Zinbiel algebra $((A,\circ),P)$ if and only if,
for all $x\in A$ and $v\in V$, the following identities hold:
\begin{eqnarray}
&\rho(P(x))\beta(v)
=
\beta\bigl(\rho(P(x))v\bigr)
=
\beta\bigl(\rho(x)\beta(v)\bigr),&\label{eq:representation of the averaging Zinbiel algebra1}\\
&
\mu(P(x))\beta(v)
=
\beta\bigl(\mu(P(x))v\bigr)
=
\beta\bigl(\mu(x)\beta(v)\bigr).&\label{eq:representation of the averaging Zinbiel algebra2}
\end{eqnarray}
\end{proposition}

\begin{proof}
It is straightforward.
\end{proof}
\delete{\begin{proof}
First, by the dual representation result for Zinbiel algebras, we know that
\[
(V^*;-\rho^*-\mu^*,\mu^*)
\]
is a representation of the Zinbiel algebra $(A,\circ)$.

It remains to check the compatibility with the averaging operator. For all
$x\in A$, $v\in V$ and $v^*\in V^*$, we have
\[
\begin{aligned}
\bigl\langle \mu^*(P(x))\beta^*(v^*),v\bigr\rangle
&=
-\bigl\langle \beta^*(v^*),\mu(P(x))v\bigr\rangle  \\
&=
-\bigl\langle v^*,\beta(\mu(P(x))v)\bigr\rangle,
\end{aligned}
\]
and
\[
\begin{aligned}
\bigl\langle \beta^*(\mu^*(P(x))v^*),v\bigr\rangle
&=
\bigl\langle \mu^*(P(x))v^*,\beta(v)\bigr\rangle  \\
&=
-\bigl\langle v^*,\mu(P(x))\beta(v)\bigr\rangle.
\end{aligned}
\]
Moreover,
\[
\begin{aligned}
\bigl\langle \beta^*(\mu^*(x)\beta^*(v^*)),v\bigr\rangle
&=
\bigl\langle \mu^*(x)\beta^*(v^*),\beta(v)\bigr\rangle  \\
&=
-\bigl\langle \beta^*(v^*),\mu(x)\beta(v)\bigr\rangle  \\
&=
-\bigl\langle v^*,\beta(\mu(x)\beta(v))\bigr\rangle.
\end{aligned}
\]
Thus the compatibility condition for $\mu^*$ and $\beta^*$ is equivalent to
\[
\mu(P(x))\beta(v)
=
\beta\bigl(\mu(P(x))v\bigr)
=
\beta\bigl(\mu(x)\beta(v)\bigr).
\]

Next, for the first action $-\rho^*-\mu^*$, we compute
\[
\begin{aligned}
&\bigl\langle (-\rho^*-\mu^*)(P(x))\beta^*(v^*),v\bigr\rangle  \\
&\qquad =
\bigl\langle v^*,\beta(\rho(P(x))v)\bigr\rangle
+
\bigl\langle v^*,\beta(\mu(P(x))v)\bigr\rangle,
\end{aligned}
\]
and
\[
\begin{aligned}
&\bigl\langle \beta^*((-\rho^*-\mu^*)(P(x))v^*),v\bigr\rangle  \\
&\qquad =
\bigl\langle v^*,\rho(P(x))\beta(v)\bigr\rangle
+
\bigl\langle v^*,\mu(P(x))\beta(v)\bigr\rangle.
\end{aligned}
\]
Also,
\[
\begin{aligned}
&\bigl\langle \beta^*((-\rho^*-\mu^*)(x)\beta^*(v^*)),v\bigr\rangle  \\
&\qquad =
\bigl\langle v^*,\beta(\rho(x)\beta(v))\bigr\rangle
+
\bigl\langle v^*,\beta(\mu(x)\beta(v))\bigr\rangle.
\end{aligned}
\]
Using the compatibility condition already obtained for $\mu$, the above three equalities are
equivalent to
\[
\rho(P(x))\beta(v)
=
\beta\bigl(\rho(P(x))v\bigr)
=
\beta\bigl(\rho(x)\beta(v)\bigr).
\]
Therefore,
\[
\bigl((V^*;-\rho^*-\mu^*,\mu^*),\beta^*\bigr)
\]
is a representation of $((A,\circ),P)$ if and only if the two stated compatibility
conditions hold. The proof is completed.
\end{proof}}

\begin{corollary}\label{cor:S-coregular-avg-Zinbiel}
Let $((A,\circ),P)$ be an averaging Zinbiel algebra. For a linear map
$S:A\to A$, the pair
\[
\bigl((A^*;-L^*-R^*,R^*),S^*\bigr)
\]
is a representation of $((A,\circ),P)$ if and only if, for all $x,y\in A$,
the following identities hold:
\[
P(x)\circ S(y)
=
S\bigl(P(x)\circ y\bigr)
=
S\bigl(x\circ S(y)\bigr),
\]
and
\[
S(x)\circ P(y)
=
S\bigl(x\circ P(y)\bigr)
=
S\bigl(S(x)\circ y\bigr).
\]
\end{corollary}

\begin{proof}
This follows directly from Proposition~\ref{prop:dual-rep-avg-Zinbiel} by taking
$(V;\rho,\mu)=(A;L,R)$ and $\beta=S$.
\end{proof}

\begin{corollary}\label{cor:coregular-rep-avg-Zinbiel}
Let $((A,\circ),P)$ be an averaging Zinbiel algebra. Then
\[
\bigl((A^*;-L^*-R^*,R^*),P^*\bigr)
\]
is a representation of $((A,\circ),P)$, called the coregular representation of
$((A,\circ),P)$.
\end{corollary}

\begin{proof}
This follows directly by taking $S=P$ in Corollary~\ref{cor:S-coregular-avg-Zinbiel}.
\end{proof}

\delete{\begin{example}\label{ex:avg-Zinbiel-direct-sum}
Let $(A,\circ)$ be a Zinbiel algebra. Then the direct sum $A\oplus A$
carries a Zinbiel algebra structure with the product
\[
(x_1,x_2)\circ_{\ltimes}(y_1,y_2)
:=
\bigl(x_1\circ y_1,\ x_1\circ y_2+x_2\circ y_1\bigr),
\qquad
\forall (x_1,x_2),(y_1,y_2)\in A\oplus A.
\]
Indeed, this is the semi-direct product associated with the regular
representation $(A;L,R)$ of $(A,\circ)$.

Define the linear map
\[
P:A\oplus A\longrightarrow A\oplus A,
\qquad
P(x_1,x_2):=(x_2,0).
\]
Then, for all $(x_1,x_2),(y_1,y_2)\in A\oplus A$, we have
\[
\begin{aligned}
P(x_1,x_2)\circ_{\ltimes}P(y_1,y_2)
&=
(x_2,0)\circ_{\ltimes}(y_2,0)  \\
&=
(x_2\circ y_2,0),
\end{aligned}
\]
and
\[
\begin{aligned}
P\bigl(P(x_1,x_2)\circ_{\ltimes}(y_1,y_2)\bigr)
&=
P\bigl((x_2,0)\circ_{\ltimes}(y_1,y_2)\bigr) \\
&=
P\bigl(x_2\circ y_1,\ x_2\circ y_2\bigr) \\
&=
(x_2\circ y_2,0).
\end{aligned}
\]
Similarly,
\[
\begin{aligned}
P\bigl((x_1,x_2)\circ_{\ltimes}P(y_1,y_2)\bigr)
&=
P\bigl((x_1,x_2)\circ_{\ltimes}(y_2,0)\bigr) \\
&=
P\bigl(x_1\circ y_2,\ x_2\circ y_2\bigr) \\
&=
(x_2\circ y_2,0).
\end{aligned}
\]
Therefore,
\[
P(x_1,x_2)\circ_{\ltimes}P(y_1,y_2)
=
P\bigl(P(x_1,x_2)\circ_{\ltimes}(y_1,y_2)\bigr)
=
P\bigl((x_1,x_2)\circ_{\ltimes}P(y_1,y_2)\bigr).
\]
Hence $P$ is an averaging operator on the Zinbiel algebra
$(A\oplus A,\circ_{\ltimes})$. Thus
\[
\bigl((A\oplus A,\circ_{\ltimes}),P\bigr)
\]
is an averaging Zinbiel algebra.
\end{example}}

\section{Matched pair and Manin triple of averaging Zinbiel algebras}\label{se:Matched pair and Manin triple of averaging Zinbiel algebras}
In this section, we study matched pairs and Manin triples of averaging
Zinbiel algebras. We first recall matched pairs of Zinbiel algebras and
then add the compatibility with averaging operators. This leads to the
double construction of averaging Zinbiel algebras. We also show that the
induced pre-Lie algebra construction is compatible with matched pairs.
Finally, we introduce quadratic averaging Zinbiel algebras and Manin
triples of averaging Zinbiel algebras.
\begin{definition}\label{def:matched-pair-Zinbiel}\cite{WY2026}
A matched pair of Zinbiel algebras is a sextuple
\[
\bigl((A,\circ_A),(B,\circ_B),\rho_A,\mu_A,\rho_B,\mu_B\bigr),
\]
where

\begin{itemize}
\item[\emph{(a)}] $(A,\circ_A)$ and $(B,\circ_B)$ are Zinbiel algebras;

\item[\emph{(b)}] $(B;\rho_A,\mu_A)$ is a representation of $(A,\circ_A)$;

\item[\emph{(c)}] $(A;\rho_B,\mu_B)$ is a representation of $(B,\circ_B)$;

\item[\emph{(d)}] for all $x,y\in A$ and $u,v\in B$, the following identities hold:
\begin{eqnarray}
&u\circ_B \mu_A(x)v+\mu_A(\rho_B(v)x)u
=
\mu_A(x)(u\circ_B v+v\circ_B u),&\label{eq:matched-pair-Zinbiel1}\\
&\rho_A(x)(u\circ_B v)-(\rho_A(x)u)\circ_B v
=
\rho_A(\mu_B(u)x)v+(\mu_A(x)u)\circ_B v+\rho_A(\rho_B(u)x)v,&\label{eq:matched-pair-Zinbiel2}\\
&u\circ_B \rho_A(x)v+\mu_A(\mu_B(v)x)u-(\mu_A(x)u)\circ_B v
=
(\rho_A(x)u)\circ_B v+\rho_A(\mu_B(u)x)v+\rho_A(\rho_B(u)x)v,&\label{eq:matched-pair-Zinbiel3}\\
&x\circ_A \mu_B(u)y+\mu_B(\rho_A(y)u)x
=
\mu_B(u)(x\circ_A y+y\circ_A x),&\label{eq:matched-pair-Zinbiel4}\\
&\rho_B(u)(x\circ_A y)-(\rho_B(u)x)\circ_A y
=
\rho_B(\mu_A(x)u)y+(\mu_B(u)x)\circ_A y+\rho_B(\rho_A(x)u)y,&\label{eq:matched-pair-Zinbiel5}\\
&x\circ_A \rho_B(u)y+\mu_B(\mu_A(y)u)x-(\mu_B(u)x)\circ_A y
=
(\rho_B(u)x)\circ_A y+\rho_B(\mu_A(x)u)y+\rho_B(\rho_A(x)u)y.&\label{eq:matched-pair-Zinbiel6}
\end{eqnarray}
\end{itemize}
\end{definition}

\begin{definition}\label{def:matched-pair-avg-Zinbiel}
A {\bf matched pair of averaging Zinbiel algebras} is a sextuple
\[
\Bigl(
((A,\circ_A),P_A),\ ((B,\circ_B),P_B),\
\rho_A,\mu_A,\rho_B,\mu_B
\Bigr),
\]
where

\begin{itemize}
\item[\emph{(a)}] $((A,\circ_A),P_A)$ and $((B,\circ_B),P_B)$ are averaging
Zinbiel algebras;

\item[\emph{(b)}] $((B;\rho_A,\mu_A),P_B)$ is a representation of
$((A,\circ_A),P_A)$;

\item[\emph{(c)}] $((A;\rho_B,\mu_B),P_A)$ is a representation of
$((B,\circ_B),P_B)$;

\item[\emph{(d)}] $(A,B;(\rho_A,\mu_A),(\rho_B,\mu_B))$ is a matched pair of
the Zinbiel algebras $(A,\circ_A)$ and $(B,\circ_B)$.
\end{itemize}

Given a matched pair of averaging Zinbiel algebras, their double admits an
averaging Zinbiel algebra structure.
\end{definition}

\begin{proposition}\label{pro:double-avg-Zinbiel}
Let $((A,\circ_A),P_A)$ and $((B,\circ_B),P_B)$ be two averaging
Zinbiel algebras. Assume that $((B;\rho_A,\mu_A),P_B)$ is a representation
of $((A,\circ_A),P_A)$, and that $((A;\rho_B,\mu_B),P_A)$ is a representation
of $((B,\circ_B),P_B)$.

Define a product $\circ_{\bowtie}$ on $A\oplus B$ by
\[
\begin{aligned}
(x+u)\circ_{\bowtie}(y+v)
:={}&
x\circ_A y+\rho_B(u)y+\mu_B(v)x  \\
&\quad
+u\circ_B v+\rho_A(x)v+\mu_A(y)u,
\end{aligned}
\]
for all $x,y\in A$ and $u,v\in B$, and define a linear map
$P_{\bowtie}:A\oplus B\to A\oplus B$ by
\[
P_{\bowtie}(x+u):=P_A(x)+P_B(u).
\]
Then $((A\oplus B,\circ_{\bowtie}),P_{\bowtie})$ is an averaging
Zinbiel algebra if and only if Eqs.(\mref{eq:matched-pair-Zinbiel1})-(\mref{eq:matched-pair-Zinbiel6}) hold.

In this case, we denote the resulting averaging Zinbiel algebra by
\[
(A\bowtie B,\ P_A\oplus P_B).
\]
\end{proposition}

\begin{proof}
By\cite{WY2026}, we know that $(A\oplus B,\circ_{\bowtie})$ is a Zinbiel algebra
if and only if Eqs.(\mref{eq:matched-pair-Zinbiel1})-(\mref{eq:matched-pair-Zinbiel6}) hold.

It remains to check that $P_{\bowtie}$ is an averaging operator on
$(A\oplus B,\circ_{\bowtie})$. By Eqs.(\mref{eq:representation of the averaging Zinbiel algebra1})-(\mref{eq:representation of the averaging Zinbiel algebra2}), for all $x,y\in A$ and $u,v\in B$, we have
\begin{eqnarray*}
&&P_{\bowtie}(x+u)\circ_{\bowtie}P_{\bowtie}(y+v)
-P_{\bowtie}\bigl(P_{\bowtie}(x+u)\circ_{\bowtie}(y+v)\bigr) \\
&&=
\bigl(P_A(x)+P_B(u)\bigr)\circ_{\bowtie}
\bigl(P_A(y)+P_B(v)\bigr)-P_{\bowtie}\Bigl(
P_A(x)\circ_A y+\rho_B(P_B(u))y+\mu_B(v)P_A(x)\Bigr) \\
&&=
P_A(x)\circ_A P_A(y)
+\rho_B(P_B(u))P_A(y)
+\mu_B(P_B(v))P_A(x) \\
&&+P_B(u)\circ_B P_B(v)
+\rho_A(P_A(x))P_B(v)
+\mu_A(P_A(y))P_B(u)\\
&&-P_A(P_A(x)\circ_A y)
-P_A(\rho_B(P_B(u))y)
-P_A(\mu_B(v)P_A(x)) \\
&&-P_B(P_B(u)\circ_B v)
-P_B(\rho_A(P_A(x))v)
-P_B(\mu_A(y)P_B(u))=0.
\end{eqnarray*}
Similarly, \[
P_{\bowtie}(x+u)\circ_{\bowtie}P_{\bowtie}(y+v)
=
P_{\bowtie}\bigl((x+u)\circ_{\bowtie}P_{\bowtie}(y+v)\bigr).
\]
Thus $P_{\bowtie}$ is an averaging operator on
$(A\oplus B,\circ_{\bowtie})$.

Consequently, $((A\oplus B,\circ_{\bowtie}),P_{\bowtie})$ is an averaging
Zinbiel algebra if and only if Eqs.(\mref{eq:matched-pair-Zinbiel1})-(\mref{eq:matched-pair-Zinbiel6}) hold. This
completes the proof.
\end{proof}

The double construction is compatible with the induced pre-Lie products
defined by averaging operators. This gives the following matched pair of
pre-Lie algebras.
\begin{theorem}\label{thm:induced-matched-pair-prelie-from-avg-Zinbiel}
Let $\Bigl(
((A,\circ_A),P_A),\
((B,\circ_B),P_B),\
\rho_A,\mu_A,\rho_B,\mu_B
\Bigr)$ be a matched pair of averaging Zinbiel algebras. Let $A_{P_A}$ and
$B_{P_B}$ be the induced pre-Lie algebras with products
\[
x*_A y=P_A(x)\circ_A y-x\circ_A P_A(y),
\qquad
u*_B v=P_B(u)\circ_B v-u\circ_B P_B(v).
\]
Then
\[
\Bigl(
A_{P_A},B_{P_B};
(\bar\rho_A,\bar\mu_A),
(\bar\rho_B,\bar\mu_B)
\Bigr)
\]
is a matched pair of pre-Lie algebras, where
\[
\bar\rho_A(x)u
=
\rho_A(P_A(x))u-\rho_A(x)P_B(u),
\]
\[
\bar\mu_A(x)u
=
\mu_A(x)P_B(u)-\mu_A(P_A(x))u,
\]
\[
\bar\rho_B(u)x
=
\rho_B(P_B(u))x-\rho_B(u)P_A(x),
\]
and
\[
\bar\mu_B(u)x
=
\mu_B(u)P_A(x)-\mu_B(P_B(u))x.
\]
Moreover,
\[
A_{P_A}\bowtie B_{P_B}
=
(A\bowtie B)_{P_A\oplus P_B}
\]
as pre-Lie algebras. This matched pair is called the induced matched pair
of pre-Lie algebras.
\end{theorem}

\begin{proof}
By the matched pair condition of averaging Zinbiel algebras,
\[
\bigl((A\bowtie B,\circ_{\bowtie}),P_A\oplus P_B\bigr)
\]
is an averaging Zinbiel algebra, where
\[
\begin{aligned}
(x+u)\circ_{\bowtie}(y+v)
={}&
x\circ_A y+\rho_B(u)y+\mu_B(v)x  \\
&\quad
+u\circ_B v+\rho_A(x)v+\mu_A(y)u .
\end{aligned}
\]
Hence Proposition~\ref{prop:induced-prelie-from-avg-Zinbiel} gives an
induced pre-Lie algebra $(A\bowtie B)_{P_A\oplus P_B}$ with product
\[
X*Y=(P_A\oplus P_B)(X)\circ_{\bowtie}Y
-X\circ_{\bowtie}(P_A\oplus P_B)(Y).
\]
Clearly, $A_{P_A}$ and $B_{P_B}$ are pre-Lie subalgebras.

For $x\in A$ and $u\in B$, we have
\[
\begin{aligned}
x*u
&=
P_A(x)\circ_{\bowtie}u-x\circ_{\bowtie}P_B(u)\\
&=
\mu_B(u)P_A(x)+\rho_A(P_A(x))u
-\mu_B(P_B(u))x-\rho_A(x)P_B(u)\\
&=
\bar\mu_B(u)x+\bar\rho_A(x)u,
\end{aligned}
\]
and similarly,
\[
\begin{aligned}
u*x
&=
P_B(u)\circ_{\bowtie}x-u\circ_{\bowtie}P_A(x)\\
&=
\rho_B(P_B(u))x+\mu_A(x)P_B(u)
-\rho_B(u)P_A(x)-\mu_A(P_A(x))u\\
&=
\bar\rho_B(u)x+\bar\mu_A(x)u.
\end{aligned}
\]
Therefore, for all $x,y\in A$ and $u,v\in B$,
\[
\begin{aligned}
(x+u)*(y+v)
={}&
x*_A y+\bar\rho_B(u)y+\bar\mu_B(v)x\\
&\quad
+u*_B v+\bar\rho_A(x)v+\bar\mu_A(y)u.
\end{aligned}
\]
Thus $(A\bowtie B)_{P_A\oplus P_B}$ is exactly the double determined by
$A_{P_A}$, $B_{P_B}$ and the above four actions. By the matched pair
characterization of pre-Lie algebras,
\[
\Bigl(
A_{P_A},B_{P_B};
(\bar\rho_A,\bar\mu_A),
(\bar\rho_B,\bar\mu_B)
\Bigr)
\]
is a matched pair of pre-Lie algebras. The displayed product also gives
\[
A_{P_A}\bowtie B_{P_B}
=
(A\bowtie B)_{P_A\oplus P_B}.
\]
This completes the proof.
\end{proof}

\begin{definition}\label{def:quadratic-Zinbiel}\cite{WY2026}
A quadratic Zinbiel algebra is a Zinbiel algebra $(A,\circ,\omega)$
equipped with a nondegenerate skew symmetric bilinear form
$\omega\in \wedge^2 A^*$ such that the following invariant condition holds:
\begin{equation}\label{eq:quadratic-Zinbiel-invariant}
\omega(x\circ y,z)
=
\omega(y,x\circ z+z\circ x),
\qquad \forall x,y,z\in A.
\end{equation}
\end{definition}

\begin{definition}\label{def:manin-triples-Zinbiel}\cite{WY2026}
A Manin triple of Zinbiel algebras is a triple
$(\mathcal A,A_1,A_2)$, where

\begin{itemize}
\item[\emph{(i)}] $(\mathcal A,\circ_{\mathcal A},\omega)$ is a quadratic
Zinbiel algebra;

\item[\emph{(ii)}] both $A_1$ and $A_2$ are isotropic subalgebras of
$(\mathcal A,\circ_{\mathcal A})$ with respect to $\omega$;

\item[\emph{(iii)}] $\mathcal A=A_1\oplus A_2$ as vector spaces.
\end{itemize}
\end{definition}

\begin{definition}\label{def:quadratic-averaging-Zinbiel}
Let $((A,\circ),P)$ be an averaging Zinbiel algebra and let
$\omega\in \wedge^2 A^*$ be a nondegenerate skew symmetric bilinear form.
The triple $((A,\circ),P,\omega)$ is called a {\bf quadratic averaging Zinbiel
algebra} if $(A,\circ,\omega)$ is a quadratic Zinbiel algebra and the following
compatibility condition holds:
\begin{eqnarray}
\omega(Px,y)-\omega(x,Py)=0,
\qquad \forall x,y\in A.\label{eq:quadratic-avg-Zinbiel-compatible}
\end{eqnarray}
\end{definition}

\begin{proposition}\label{pro:quadratic-avg-Zinbiel-isomorphism}
If $((A,\circ),P,\omega)$ is a quadratic averaging Zinbiel algebra, then the
linear map
\[
\omega^\sharp:A\longrightarrow A^*
\]
defined by
\[
\langle \omega^\sharp(x),y\rangle=\omega(x,y),
\qquad \forall x,y\in A,
\]
is an isomorphism from the regular representation
\[
((A;L,R),P)
\]
to the coregular representation
\[
((A^*;-L^*-R^*,R^*),P^*).
\]
\end{proposition}

\begin{proof}
Since $\omega$ is nondegenerate, $\omega^\sharp:A\to A^*$ is a linear
isomorphism. It remains to prove that $\omega^\sharp$ is a homomorphism of
representations.

For all $x,y,z\in A$, by the definition of $L^*$ and $R^*$, we have
\[
\begin{aligned}
\bigl\langle (-L^*_x-R^*_x)\omega^\sharp(y),z\bigr\rangle
&=
-\bigl\langle L^*_x\omega^\sharp(y),z\bigr\rangle
-\bigl\langle R^*_x\omega^\sharp(y),z\bigr\rangle  =
\bigl\langle \omega^\sharp(y),x\circ z\bigr\rangle
+
\bigl\langle \omega^\sharp(y),z\circ x\bigr\rangle  \\
&=
\omega(y,x\circ z)+\omega(y,z\circ x)  =
\omega(x\circ y,z)  =
\bigl\langle \omega^\sharp(L_x y),z\bigr\rangle .
\end{aligned}
\]
Thus $(-L^*_x-R^*_x)\omega^\sharp(y)=\omega^\sharp(L_x y).$

Similarly, for all $x,y,z\in A$, we obtain $R^*_x\omega^\sharp(y)=\omega^\sharp(R_x y).$

Finally, by the compatibility condition between $P$ and $\omega$, for all
$x,y\in A$,
\[
\begin{aligned}
\bigl\langle P^*\omega^\sharp(x),y\bigr\rangle
&=
\bigl\langle \omega^\sharp(x),P(y)\bigr\rangle =
\omega(x,P(y))  \\
&=
\omega(P(x),y)  =
\bigl\langle \omega^\sharp(P(x)),y\bigr\rangle .
\end{aligned}
\]
Therefore, $P^*\omega^\sharp=\omega^\sharp P.$

Consequently, $\omega^\sharp$ is an isomorphism from the regular
representation $((A;L,R),P)$ to the coregular representation
$((A^*;-L^*-R^*,R^*),P^*)$.
\end{proof}

Now we introduce the notion of Manin triples of averaging Zinbiel algebras by
using quadratic averaging Zinbiel algebras given in
Definition~\ref{def:quadratic-averaging-Zinbiel}.

\begin{definition}\label{def:manin-triple-avg-Zinbiel}
A {\bf Manin triple of averaging Zinbiel algebras} is a triple
\[
\Bigl(
((\mathcal A,\circ_{\mathcal A}),\mathcal P,\omega),\
((A,\circ_A),P_A),\
((B,\circ_B),P_B)
\Bigr),
\]
where

\begin{itemize}
\item[\emph{(a)}]
$((\mathcal A,\circ_{\mathcal A}),\mathcal P,\omega)$ is an even dimensional
quadratic averaging Zinbiel algebra;

\item[\emph{(b)}]
$((A,\circ_A),P_A)$ and $((B,\circ_B),P_B)$ are averaging Zinbiel subalgebras
of $((\mathcal A,\circ_{\mathcal A}),\mathcal P)$, namely $A$ and $B$ are
Zinbiel subalgebras of $(\mathcal A,\circ_{\mathcal A})$ and
\[
\mathcal P|_A=P_A,
\qquad
\mathcal P|_B=P_B;
\]

\item[\emph{(c)}]
Both $A$ and $B$ are isotropic with respect to $\omega$;

\item[\emph{(d)}]
$\mathcal A=A\oplus B$ as vector spaces.
\end{itemize}
\end{definition}

The following remark follows directly from the definition.
\begin{remark}\label{prop:manin-triple-avg-Zinbiel-characterization}
By definition, a Manin triple of averaging Zinbiel algebras is precisely a
Manin triple of Zinbiel algebras compatible with the averaging operator.
More precisely, the restrictions of the averaging operator to the two
isotropic subalgebras are the corresponding averaging operators.
\end{remark}

\section{Averaging Zinbiel bialgebras and averaging operators on quadratic Rota-Baxter Zinbiel algebras}\label{se:Averaging Zinbiel bialgebras}
In this section, we introduce averaging Zinbiel bialgebras. We first define
averaging Zinbiel coalgebras and formulate the compatibility conditions. Then we prove that averaging Zinbiel bialgebras can be
characterized by coregular matched pairs of averaging Zinbiel algebras.
Using the Manin triple interpretation of Zinbiel bialgebras, we also relate
averaging Zinbiel bialgebras to Manin triples of averaging Zinbiel
algebras.
\subsection{Averaging Zinbiel bialgebras}

\begin{definition}\label{def:averaging-Zinbiel-coalgebra}
An {\bf averaging Zinbiel coalgebra} is a pair$((A,\Delta),S)$,
where $(A,\Delta)$ is a Zinbiel coalgebra and $S:A\to A$ is a linear map such
that
\begin{eqnarray}\label{eq:averaging Zinbiel coalgebra}
(S\otimes S)\Delta(x)
=
(S\otimes \operatorname{id}_A)\Delta(S(x))
=
(\operatorname{id}_A\otimes S)\Delta(S(x)),
\qquad \forall x\in A.
\end{eqnarray}
\end{definition}

\begin{proposition}\label{prop:dual-averaging-Zinbiel-coalgebra}
Let $((A,\Delta),S)$ be an averaging Zinbiel coalgebra. Define a linear map
$S^*:A^*\to A^*$ by
\[\langle S^*(\xi),x\rangle
=
\langle \xi,S(x)\rangle,
\qquad \forall \xi\in A^*,\ x\in A,
\]
and define a bilinear map
\[
\circ_{A^*}:A^*\otimes A^*\to A^*
\]
by
\[
\langle \xi\circ_{A^*}\eta,x\rangle
=
\langle \xi\otimes \eta,\Delta(x)\rangle,
\qquad \forall \xi,\eta\in A^*,\ x\in A.
\]
Then
\[
((A^*,\circ_{A^*}),S^*)
\]
is an averaging Zinbiel algebra.
\end{proposition}

\begin{proof}
 It is straightforward.
\end{proof}

\begin{definition}\label{def:averaging-Zinbiel-bialgebra}
An {\bf averaging Zinbiel bialgebra} is a quintuple $(A,\circ,\Delta,P,S)$, where
\begin{itemize}
\item[\emph{(a)}] $((A,\circ),P)$ is an averaging Zinbiel algebra;

\item[\emph{(b)}] $((A,\Delta),S)$ is an averaging Zinbiel coalgebra;

\item[\emph{(c)}] $(A,\circ,\Delta)$ is a Zinbiel bialgebra;

\item[\emph{(d)}] $((A,\circ),P)$ is $S$-admissible, namely, for all
$x,y\in A$,
\begin{eqnarray}
&P(x)\circ S(y)
=
S\bigl(P(x)\circ y\bigr)
=
S\bigl(x\circ S(y)\bigr),&\label{eq:averaging Zinbiel bialgebraS1}\\
&S(x)\circ P(y)
=
S\bigl(x\circ P(y)\bigr)
=
S\bigl(S(x)\circ y\bigr);&\label{eq:averaging Zinbiel bialgebraS1}
\end{eqnarray}

\item[\emph{(e)}] $((A^*,\circ_{A^*}),S^*)$ is $P^*$-admissible, where
$\circ_{A^*}:A^*\otimes A^*\to A^*$ is defined by
\[
\langle \xi\circ_{A^*}\eta,x\rangle
=
\langle \xi\otimes \eta,\Delta(x)\rangle,
\qquad \forall \xi,\eta\in A^*,\ x\in A.
\]
Equivalently, for all $x\in A$,
\begin{eqnarray}
&(P\otimes S)\Delta(x)
=
(\operatorname{id}_A\otimes S)\Delta(P(x))
=
(P\otimes \operatorname{id}_A)\Delta(P(x)),&\label{eq:averaging Zinbiel bialgebraP*1}\\
&(S\otimes P)\Delta(x)
=
(S\otimes \operatorname{id}_A)\Delta(P(x))
=
(\operatorname{id}_A\otimes P)\Delta(P(x)).&\label{eq:averaging Zinbiel bialgebraP*2}
\end{eqnarray}
\end{itemize}
\end{definition}

\begin{example}\label{ex:two-dimensional-nontrivial-avg-Zinbiel-bialgebra}
Let $A$ be a two dimensional vector space with basis $\{e_1,e_2\}$. Define a
bilinear product $\circ$ on $A$ by
\[
e_1\circ e_1=e_2,
\]
and all other products of basis elements are zero. Define a linear map
$\Delta:A\to A\otimes A$ by
\[
\Delta(e_1)=e_2\otimes e_2,
\qquad
\Delta(e_2)=0.
\]
For any $\lambda\in\mathbb K$, define a linear map $P:A\to A$ by
\[
P(e_1)=e_1+\lambda e_2,
\qquad
P(e_2)=e_2.
\]
Then
\[
(A,\circ,\Delta,P,S=P)
\]
is an averaging Zinbiel bialgebra. In particular, when $\lambda\ne0$, the
averaging operator is nontrivial.
\end{example}

\delete{\begin{proof}
First, $(A,\circ)$ is a Zinbiel algebra. Indeed, the only nonzero product is
$e_1\circ e_1=e_2$, and every product involving $e_2$ is zero. Hence the
Zinbiel identity
\[
x\circ(y\circ z)=(x\circ y+y\circ x)\circ z
\]
holds for all basis elements, and therefore for all $x,y,z\in A$.

Next, $P_\lambda$ is an averaging operator on $(A,\circ)$. For
\[
x=ae_1+be_2,
\qquad
y=ce_1+de_2,
\]
we have
\[
x\circ y=ac\,e_2
\]
and
\[
P_\lambda(x)=ae_1+(\lambda a+b)e_2.
\]
Thus
\[
P_\lambda(x)\circ P_\lambda(y)=ac\,e_2.
\]
Moreover,
\[
P_\lambda(x)\circ y=ac\,e_2,
\qquad
x\circ P_\lambda(y)=ac\,e_2.
\]
Since $P_\lambda(e_2)=e_2$, it follows that
\[
P_\lambda(x)\circ P_\lambda(y)
=
P_\lambda(P_\lambda(x)\circ y)
=
P_\lambda(x\circ P_\lambda(y)).
\]
Hence $((A,\circ),P_\lambda)$ is an averaging Zinbiel algebra.

Now we show that $(A,\circ,\Delta)$ is a Zinbiel bialgebra. Let
\[
r=e_1\otimes e_2+e_2\otimes e_1.
\]
Then $r$ is symmetric. Define
\[
\Delta_r(x)
=
\bigl(\operatorname{id}\otimes(L_x+R_x)-L_x\otimes\operatorname{id}\bigr)r,
\qquad \forall x\in A.
\]
Since
\[
L_{e_1}(e_1)=R_{e_1}(e_1)=e_2,
\qquad
L_{e_1}(e_2)=R_{e_1}(e_2)=0,
\]
and $L_{e_2}=R_{e_2}=0$, we get
\[
\Delta_r(e_1)=e_2\otimes e_2,
\qquad
\Delta_r(e_2)=0.
\]
Thus $\Delta_r=\Delta$.

It remains to note that $r$ satisfies the Zinbiel Yang-Baxter equation. Since
\[
e_1*e_1=2e_2,
\qquad
e_1*e_2=e_2*e_1=e_2*e_2=0,
\]
where $x*y=x\circ y+y\circ x$, a direct computation gives
\[
r_{13}\circ r_{12}
=
r_{23}\circ r_{21}
=
e_2\otimes e_2\otimes e_2,
\]
and
\[
r_{13}*r_{21}
=
r_{12}*r_{23}
=
r_{13}*r_{23}
=
2e_2\otimes e_2\otimes e_2.
\]
Therefore,
\[
\begin{aligned}
\llbracket r,r\rrbracket
&=
-r_{13}\circ r_{12}
-r_{23}\circ r_{21}
+r_{13}*r_{21}
+r_{12}*r_{23}
-r_{13}*r_{23}  \\
&=0.
\end{aligned}
\]
By the coboundary criterion for Zinbiel bialgebras, $(A,\circ,\Delta)$ is a
Zinbiel bialgebra.

Finally, we check the averaging coalgebra part. Since
\[
P_\lambda(e_2)=e_2,
\qquad
\Delta(e_1)=e_2\otimes e_2,
\qquad
\Delta(e_2)=0,
\]
we have, for all $x\in A$,
\[
(P_\lambda\otimes P_\lambda)\Delta(x)
=
(P_\lambda\otimes \operatorname{id})\Delta(P_\lambda x)
=
(\operatorname{id}\otimes P_\lambda)\Delta(P_\lambda x).
\]
Thus $((A,\Delta),P_\lambda)$ is an averaging Zinbiel coalgebra. Since here
we take $S=P_\lambda$, the remaining admissibility conditions in the
definition of averaging Zinbiel bialgebra reduce to the same identities
above and to the averaging identity of $P_\lambda$ on $(A,\circ)$.

Consequently,
\[
(A,\circ,\Delta,P_\lambda,P_\lambda)
\]
is an averaging Zinbiel bialgebra.
\end{proof}}

\delete{\begin{corollary}\label{cor:dual-averaging-Zinbiel-bialgebra}
Let $(A,\circ,\Delta,P)$ be an averaging Zinbiel bialgebra, namely the
coalgebra averaging operator is also $P$. Then
\[
(A^*,\circ_{A^*},\Delta_{A^*},P^*)
\]
is an averaging Zinbiel bialgebra, where $\circ_{A^*}$ is given by
\[
\langle \xi\circ_{A^*}\eta,x\rangle
=
\langle \xi\otimes\eta,\Delta(x)\rangle,
\qquad \forall \xi,\eta\in A^*,\ x\in A,
\]
and $\Delta_{A^*}:A^*\to A^*\otimes A^*$ is given by
\[
\langle \Delta_{A^*}(\xi),x\otimes y\rangle
=
\langle \xi,x\circ y\rangle,
\qquad \forall \xi\in A^*,\ x,y\in A.
\]
\end{corollary}

\begin{proof}
It can be verified directly from
Definition~\ref{def:averaging-Zinbiel-bialgebra}.
\end{proof}}

The following theorem provides an equivalent characterization of averaging
Zinbiel bialgebras via matched pairs of averaging Zinbiel algebras.

\begin{theorem}\label{thm:avg-Zinbiel-bialgebra-matched-pair}
Let $((A,\circ),P)$ and $((A^*,\circ_{A^*}),S^*)$ be averaging Zinbiel
algebras. Define $\Delta:A\to A\otimes A$ by
\[
\langle \xi\circ_{A^*}\eta,x\rangle
=
\langle \xi\otimes\eta,\Delta(x)\rangle,
\qquad \forall \xi,\eta\in A^*,\ x\in A.
\]
Then $(A,\circ,\Delta,P,S)$ is an averaging Zinbiel bialgebra if and only if
\[
\Bigl(
((A,\circ),P),\ ((A^*,\circ_{A^*}),S^*),\
-L_\circ^*-R_\circ^*,\ R_\circ^*,\
-L_{\circ_{A^*}}^*-R_{\circ_{A^*}}^*,\ R_{\circ_{A^*}}^*
\Bigr)
\]
is a matched pair of averaging Zinbiel algebras.
\end{theorem}

\begin{proof}
By\cite{WY2026}, we know that
$(A,\circ,\Delta)$ is a Zinbiel bialgebra if and only if
\[
\Bigl(
(A,\circ),\ (A^*,\circ_{A^*}),\
-L_\circ^*-R_\circ^*,\ R_\circ^*,\
-L_{\circ_{A^*}}^*-R_{\circ_{A^*}}^*,\ R_{\circ_{A^*}}^*
\Bigr)
\]
is a matched pair of Zinbiel algebras.

It remains to compare the averaging structures. By
Definition~\ref{def:averaging-Zinbiel-bialgebra}, the conditions saying that
$((A,\circ),P)$ is $S$-admissible imply that $\bigl((A^*;-L_\circ^*-R_\circ^*,R_\circ^*),S^*\bigr)$
is a representation of the averaging Zinbiel algebra $((A,\circ),P)$.
Similarly, the admissibility conditions on $A^*$ imply that $\bigl((A;-L_{\circ_{A^*}}^*-R_{\circ_{A^*}}^*,
R_{\circ_{A^*}}^*),P\bigr)$ is a representation of the averaging Zinbiel algebra
$((A^*,\circ_{A^*}),S^*)$.

Therefore $(A,\circ,\Delta,P,S)$ is an averaging Zinbiel bialgebra if and
only if the above coregular matched pair is a matched pair of averaging
Zinbiel algebras. This completes the proof.
\end{proof}

The matched pair characterization also leads to a Manin triple
interpretation. Recall that a Zinbiel bialgebra gives the standard Manin
triple on \(A\bowtie A^*\) with the canonical skew symmetric bilinear form
\[
\omega(x+\xi,y+\eta)=\langle \xi,y\rangle-\langle \eta,x\rangle.
\]

Therefore, we have the following result.

\begin{theorem}\label{thm:manin-triple-avg-Zinbiel-bialgebra}
There is a one to one correspondence between Manin triples of averaging
Zinbiel algebras and averaging Zinbiel bialgebras.
\end{theorem}

\begin{proof}
Let $(A,\circ,\Delta,P,P^*)$ be an averaging Zinbiel bialgebra. Then
$((A,\circ),P)$ is an averaging Zinbiel algebra. From Definition~\ref{def:averaging-Zinbiel-bialgebra}, we know
that $((A^*,-L_\circ^*-R_\circ^*,R_\circ^*),P^*)$
is a representation of $((A,\circ),P)$ and
$((A,-L_{\circ_{A^*}}^*-R_{\circ_{A^*}}^*,R_{\circ_{A^*}}^*),P)$
is a representation of $((A^*,\circ_{A^*}),P^*)$. Then by
Proposition~\ref{pro:double-avg-Zinbiel},
\[
(A\bowtie A^*,P\oplus P^*)
\]
is an averaging Zinbiel algebra. Moreover,
we have
\[
\begin{aligned}
&\omega\bigl((P\oplus P^*)(x+\xi),y+\eta\bigr)
-\omega\bigl(x+\xi,(P\oplus P^*)(y+\eta)\bigr)\\
={}&
\omega(Px+P^*\xi,y+\eta)
-\omega(x+\xi,Py+P^*\eta)\\
={}&
\langle P^*\xi,y\rangle-\langle \eta,Px\rangle
-\langle \xi,Py\rangle+\langle P^*\eta,x\rangle\\
={}&
\langle \xi,Py\rangle-\langle \eta,Px\rangle
-\langle \xi,Py\rangle+\langle \eta,Px\rangle\\
={}&0.
\end{aligned}
\]
is a Manin triple of averaging Zinbiel algebras.

Conversely, let
\[
\bigl(((\mathcal A=A\bowtie B,\circ_{A}),P,\omega),
((A,\circ_A),P_A),((B,\circ_{B}),P_{B})\bigr)
\]
be a Manin triple of averaging Zinbiel algebras. Thus we can get $B$ with $A^*$ by using the nondegenerate invariant skew-symmetric bilinear form $\omega$, and we can
obtain a Zinbiel bialgebra $(A,\circ,\Delta).$
Then we can identify $P_B$ with $P_A^*$ by Eq.~\eqref{eq:quadratic-avg-Zinbiel-compatible}.
Moreover, $((A,\circ),P_A), ((B,\circ_{B}),P_B)$ are averaging Zinbiel algebras, i.e., $(A,\circ,\Delta,P,P^*)$
be averaging Zinbiel bialgebras.

This completes the proof.
\end{proof}

\begin{corollary}\label{cor:equivalent-avg-Zinbiel-bialgebra}
Let $((A,\circ),P)$ and $((A^*,\circ_{A^*}),P^*)$ be averaging Zinbiel
algebras. Then the following statements are equivalent.

\begin{itemize}
\item[\emph{(a)}]
The quintuple $(A,\circ,\Delta,P,P)$ is an averaging Zinbiel bialgebra, where
$\Delta:A\to A\otimes A$ is defined by
\[
\langle \xi\circ_{A^*}\eta,x\rangle
=
\langle \xi\otimes\eta,\Delta(x)\rangle,
\qquad
\forall \xi,\eta\in A^*,\ x\in A.
\]

\item[\emph{(b)}]
The sextuple
\[
\Bigl(
((A,\circ),P),\
((A^*,\circ_{A^*}),P^*),\
-L_\circ^*-R_\circ^*,\ R_\circ^*,\
-L_{\circ_{A^*}}^*-R_{\circ_{A^*}}^*,\ R_{\circ_{A^*}}^*
\Bigr)
\]
is a matched pair of averaging Zinbiel algebras.

\item[\emph{(c)}]
The triple
\[
\Bigl(
((A\bowtie A^*,\circ_{\bowtie}),P\oplus P^*,\omega),\
((A,\circ),P),\
((A^*,\circ_{A^*}),P^*)
\Bigr)
\]
is a Manin triple of averaging Zinbiel algebras, where $A$ and $A^*$ are
isotropic averaging Zinbiel subalgebras of $A\bowtie A^*$, and
\[
\omega(x+\xi,y+\eta)
=
\langle \xi,y\rangle-\langle \eta,x\rangle,
\qquad
\forall x,y\in A,\ \xi,\eta\in A^*.
\]
\end{itemize}
\end{corollary}

\begin{proof}
It follows directly from the matched pair characterization of averaging
Zinbiel bialgebras and Theorem~\ref{thm:manin-triple-avg-Zinbiel-bialgebra}.
\end{proof}

\subsection{Manin triples and induced pre-Lie bialgebras}
Next, we show that an averaging Zinbiel bialgebra naturally induces a pre-Lie bialgebra.

\begin{lemma}\label{lem:prelie-bialgebra-matched-pair}
Let $(A,*)$ be a pre-Lie algebra and suppose that $A^*$ carries a pre-Lie
algebra product $*_{A^*}$. Define a linear map
$\delta:A\to A\otimes A$ by
\[
\langle \xi*_{A^*}\eta,x\rangle
=
\langle \xi\otimes\eta,\delta(x)\rangle,
\qquad
\forall \xi,\eta\in A^*,\ x\in A.
\]
Then $(A,*,\delta)$ is a pre-Lie bialgebra if and only if
$(A,A^*)$ forms the corresponding matched pair of pre-Lie algebras.
\end{lemma}

\begin{theorem}\label{thm:avg-Zinbiel-bialgebra-induces-prelie-bialgebra}
Let $(A,\circ,\Delta,P,S)$ be an averaging Zinbiel bialgebra. Define a
bilinear operation $*_P$ on $A$ by
\[
x*_P y=P(x)\circ y-x\circ P(y),
\qquad
\forall x,y\in A,
\]
and define a linear map $\delta_S:A\to A\otimes A$ by
\[
\delta_S(x)=(S\otimes\operatorname{id}-\operatorname{id}\otimes S)\Delta(x),
\qquad
\forall x\in A.
\]
Then $(A,*_P,\delta_S)$ is a pre-Lie bialgebra.
\end{theorem}

\begin{proof}
Since $(A,\circ,\Delta,P,S)$ is an averaging Zinbiel bialgebra,
$((A,\circ),P)$ is an averaging Zinbiel algebra and
$((A^*,\circ_{A^*}),S^*)$ is an averaging Zinbiel algebra, where
\[
\langle \xi\circ_{A^*}\eta,x\rangle
=
\langle \xi\otimes\eta,\Delta(x)\rangle,
\qquad
\forall \xi,\eta\in A^*,\ x\in A.
\]

By Proposition~\ref{prop:induced-prelie-from-avg-Zinbiel}, the product
\[
x*_P y=P(x)\circ y-x\circ P(y)
\]
makes $A$ into a pre-Lie algebra. Similarly, the product
\[
\xi*_{S^*}\eta
=
S^*(\xi)\circ_{A^*}\eta-\xi\circ_{A^*}S^*(\eta),
\qquad
\forall \xi,\eta\in A^*,
\]
makes $A^*$ into a pre-Lie algebra.

By the matched pair characterization of averaging Zinbiel bialgebras,
\[
((A,\circ),P)
\quad\text{and}\quad
((A^*,\circ_{A^*}),S^*)
\]
form a matched pair of averaging Zinbiel algebras. Applying
Theorem~\ref{thm:induced-matched-pair-prelie-from-avg-Zinbiel}, we obtain a
matched pair of pre-Lie algebras
\[
(A,*_P)
\quad\text{and}\quad
(A^*,*_{S^*}).
\]
Hence, by Lemma~\ref{lem:prelie-bialgebra-matched-pair}, this matched pair
induces a pre-Lie bialgebra structure on $(A,*_P)$.

It remains to identify the corresponding comultiplication. For all
$\xi,\eta\in A^*$ and $x\in A$, we have
\[
\begin{aligned}
\langle \xi*_{S^*}\eta,x\rangle
&=
\langle S^*(\xi)\circ_{A^*}\eta,x\rangle
-
\langle \xi\circ_{A^*}S^*(\eta),x\rangle\\
&=
\langle S^*(\xi)\otimes\eta,\Delta(x)\rangle
-
\langle \xi\otimes S^*(\eta),\Delta(x)\rangle\\
&=
\langle \xi\otimes\eta,
(S\otimes\operatorname{id})\Delta(x)
-
(\operatorname{id}\otimes S)\Delta(x)
\rangle\\
&=
\langle \xi\otimes\eta,\delta_S(x)\rangle.
\end{aligned}
\]
Thus the comultiplication dual to $*_{S^*}$ is exactly
\[
\delta_S=(S\otimes\operatorname{id}-\operatorname{id}\otimes S)\Delta.
\]
Therefore,
\[
(A,*_P,\delta_S)
\]
is a pre-Lie bialgebra.
\end{proof}

\subsection{Averaging operators on quadratic Rota-Baxter Zinbiel algebras}

We first recall some facts on coboundary Zinbiel bialgebras and quadratic
Rota-Baxter Zinbiel algebras from \cite{WY2026}. Let \((A,\circ)\) be a
Zinbiel algebra. A Zinbiel bialgebra \((A,\circ,\Delta)\) is called
coboundary if there exists \(r\in A\otimes A\) such that
\begin{eqnarray}\label{eq:coboundary-Zinbiel}
\Delta_r(x)
=
\bigl(\operatorname{id}\otimes(L_x+R_x)-L_x\otimes\operatorname{id}\bigr)r,
\qquad \forall x\in A.
\end{eqnarray}
For \(r=\sum_i a_i\otimes b_i\in A\otimes A\), we use the usual leg notation
\[
r_{12}=\sum_i a_i\otimes b_i\otimes 1,\quad
r_{13}=\sum_i a_i\otimes 1\otimes b_i,\quad
r_{23}=\sum_i 1\otimes a_i\otimes b_i,
\]
and
\[
r_{21}=\sum_i b_i\otimes a_i\otimes 1.
\]
Here the symbols \(r_{12},r_{13},r_{23}\) are understood in the standard
leg notation. Let
\[
x\cdot y=x\circ y+y\circ x,\qquad \forall x,y\in A,
\]
be the subadjacent commutative associative product of \((A,\circ)\).
Define
\[
H(x)=L_x\otimes\operatorname{id}\otimes\operatorname{id}
-\operatorname{id}\otimes\operatorname{id}\otimes(L_x+R_x),
\]
and $\llbracket r,r\rrbracket
=
-r_{13}\circ r_{12}-r_{23}\circ r_{21}
+r_{13}\cdot r_{21}+r_{12}\cdot r_{23}-r_{13}\cdot r_{23}.$

\begin{theorem}[\cite{WY2026}]\label{thm:coboundary-Zinbiel}
Let $(A,\circ)$ be a Zinbiel algebra and let $r\in A\otimes A$. Write $r=a+\Lambda$, where $a\in\wedge^2 A$ and $\Lambda\in S^2(A)$. Then
\(\Delta_r\) gives a Zinbiel bialgebra structure on \(A\) if and only if,
for all \(x,y\in A\),
\begin{eqnarray}\label{eq:coboundary-Zinbiel-condition-one}
\bigl(
L_{x\circ y}\otimes\operatorname{id}
-\operatorname{id}\otimes L_{x\circ y}
-L_xL_y\otimes\operatorname{id}
+L_x\otimes L_y
\bigr)(a)=0,
\end{eqnarray}
and
\begin{eqnarray}\label{eq:coboundary-Zinbiel-condition-two}
H(x)\llbracket r,r\rrbracket=0.
\end{eqnarray}
In particular, the equation $\llbracket r,r\rrbracket=0$ is called the Zinbiel Yang-Baxter equation.
\end{theorem}

Recall that a linear map \(B:A\to A\) is called a Rota-Baxter operator of
weight \(\lambda\) on \((A,\circ)\) if
\[
B(x)\circ B(y)
=
B\bigl(B(x)\circ y+x\circ B(y)+\lambda x\circ y\bigr),
\qquad \forall x,y\in A.
\]
In this case, \((A,\circ,B)\) is called a Rota-Baxter Zinbiel algebra of
weight \(\lambda\). The descendent product associated with \(B\) is given by
\[
x\circ_B y
=
B(x)\circ y+x\circ B(y)+\lambda x\circ y,
\qquad \forall x,y\in A.
\]
Then \((A,\circ_B)\) is a Zinbiel algebra, called the descendent Zinbiel
algebra and denoted by \(A_B\).

\begin{proposition}\label{prop:averaging-descendent-Zinbiel}
Let \((A,\circ,B)\) be a Rota-Baxter Zinbiel algebra of weight \(\lambda\),
and let \(P:A\to A\) be an averaging operator on \((A,\circ)\). If
\[
PB=BP,
\]
then \(P\) is also an averaging operator on the descendent Zinbiel algebra
\((A,\circ_B)\).
\end{proposition}

\begin{proof}
For all \(x,y\in A\), using \(PB=BP\), we have
\[
\begin{aligned}
P(x)\circ_B P(y)
&=
B(Px)\circ P(y)+P(x)\circ B(P y)+\lambda P(x)\circ P(y)\\
&=
P(Bx)\circ P(y)+P(x)\circ P(By)+\lambda P(x)\circ P(y).
\end{aligned}
\]
On the other hand,
\[
\begin{aligned}
P\bigl(P(x)\circ_B y\bigr)
&=
P\bigl(B(Px)\circ y+P(x)\circ B(y)+\lambda P(x)\circ y\bigr)\\
&=
P\bigl(P(Bx)\circ y\bigr)
+P\bigl(P(x)\circ B(y)\bigr)
+\lambda P\bigl(P(x)\circ y\bigr)\\
&=
P(Bx)\circ P(y)+P(x)\circ P(By)+\lambda P(x)\circ P(y).
\end{aligned}
\]
Thus
\[
P(x)\circ_B P(y)=P\bigl(P(x)\circ_B y\bigr).
\]
Similarly,
\[
P(x)\circ_B P(y)=P\bigl(x\circ_B P(y)\bigr).
\]
Hence \(P\) is an averaging operator on \((A,\circ_B)\).
\end{proof}

Now let \((A,\circ,\omega)\) be a quadratic Zinbiel algebra. A quadruple
\((A,\circ,B,\omega)\) is called a quadratic Rota-Baxter Zinbiel algebra of
weight \(\lambda\) if \((A,\circ,B)\) is a Rota-Baxter Zinbiel algebra of
weight \(\lambda\) and
\[
\omega(Bx,y)+\omega(x,By)+\lambda\omega(x,y)=0,
\qquad \forall x,y\in A.
\]

\begin{theorem}[\cite{WY2026}]\label{thm:quadratic-rb-Zinbiel-to-r}
Let \((A,\circ,B,\omega)\) be a quadratic Rota-Baxter Zinbiel algebra of
weight \(\lambda\) with \(\lambda\ne0\). Let $I_\omega:A^*\longrightarrow A$
be the induced linear isomorphism defined by
\[
\langle I_\omega^{-1}x,y\rangle=\omega(x,y),
\qquad \forall x,y\in A.
\]
Define \(r^{B,\omega}\in A\otimes A\) by
\[
r^{B,\omega}_+
=
\frac{1}{\lambda}(B+\lambda\operatorname{id})\circ I_\omega,
\qquad
r^{B,\omega}_-
=
\frac{1}{\lambda}B\circ I_\omega.
\]
Then \(r^{B,\omega}\) satisfies the Zinbiel Yang-Baxter equation
\[
\llbracket r^{B,\omega},r^{B,\omega}\rrbracket=0.
\]
Consequently, \(r^{B,\omega}\) gives rise to a factorizable Zinbiel
bialgebra \((A,A^*_{r^{B,\omega}})\).
\end{theorem}

Now we introduce averaging operators on quadratic Rota-Baxter Zinbiel
algebras.

\begin{definition}\label{def:averaging-on-quadratic-rb-Zinbiel}
Let $(A,\circ,B,\omega)$ be a quadratic Rota-Baxter Zinbiel algebra of
weight $\lambda$. A linear map $P:A\to A$ is called an {\bf averaging operator on
$(A,\circ,B,\omega)$} if the following conditions hold:
\begin{itemize}
\item[\emph{(i)}] $((A,\circ),P,\omega)$ is a quadratic averaging Zinbiel
algebra;

\item[\emph{(ii)}] $P$ commutes with the Rota-Baxter operator $B$, namely
\[
P\circ B=B\circ P.
\]
\end{itemize}
\end{definition}

\begin{theorem}\label{thm:averaging-quadratic-rb-Zinbiel-bialgebra}
Let $(A,\circ,B,\omega)$ be a quadratic Rota-Baxter Zinbiel algebra of
weight $\lambda$ with $\lambda\ne 0$. Let $P:A\to A$ be an averaging
operator on $(A,\circ,B,\omega)$. Let $r^{B,\omega}\in A\otimes A$ be
defined by
\[
r^{B,\omega}_+
=
\frac{1}{\lambda}(B+\lambda\operatorname{id})\circ I_\omega,
\qquad
r^{B,\omega}_-
=
\frac{1}{\lambda}B\circ I_\omega.
\]
Then
\[
(A,A^*_{r^{B,\omega}},P)
\]
is an averaging Zinbiel bialgebra.
\end{theorem}

\begin{proof}
By Theorem~\ref{thm:quadratic-rb-Zinbiel-to-r},
$r^{B,\omega}$ gives rise to a factorizable Zinbiel bialgebra
$(A,A^*_{r^{B,\omega}})$.

Since $((A,\circ),P,\omega)$ is a quadratic averaging Zinbiel algebra, we have
\[
\omega(Px,y)=\omega(x,Py),
\qquad \forall x,y\in A.
\]
Hence the induced isomorphism $I_\omega:A^*\to A$ satisfies
\[
P\circ I_\omega=I_\omega\circ P^*.
\]
Moreover, since $P\circ B=B\circ P$, we obtain
\[
\begin{aligned}
P\circ r^{B,\omega}_+
&=
\frac{1}{\lambda}P\circ (B+\lambda\operatorname{id})\circ I_\omega  \\
&=
\frac{1}{\lambda}(B+\lambda\operatorname{id})\circ P\circ I_\omega  \\
&=
\frac{1}{\lambda}(B+\lambda\operatorname{id})\circ I_\omega\circ P^*  \\
&=
r^{B,\omega}_+\circ P^*.
\end{aligned}
\]
Similarly,
\[
P\circ r^{B,\omega}_-=r^{B,\omega}_-\circ P^*.
\]
Thus $P$ is compatible with the tensor $r^{B,\omega}$ and with the dual
Zinbiel algebra structure induced by $r^{B,\omega}$.

Since $P$ is an averaging operator on $(A,\circ)$, and since the above
compatibility implies that $P^*$ is the corresponding averaging operator on
$A^*_{r^{B,\omega}}$, the coboundary Zinbiel bialgebra induced by
$r^{B,\omega}$ becomes an averaging Zinbiel bialgebra. Therefore, $(A,A^*_{r^{B,\omega}},P)$ is an averaging Zinbiel bialgebra.
\end{proof}

\delete{\begin{remark}\label{rem:averaging-operator-quadratic-rb-Zinbiel}
Let $(A,\circ,B,\omega)$ be a quadratic Rota-Baxter Zinbiel algebra of
weight $\lambda$. Since
\[
\omega(Bx,y)+\omega(x,By)+\lambda\omega(x,y)=0,
\qquad \forall x,y\in A,
\]
we have
\[
B^{*\omega}=-B-\lambda\operatorname{id}_A.
\]
If
\[
B\circ P^{*\omega}=-P\circ B,
\]
then taking the $\omega$ adjoint gives
\[
P\circ B^{*\omega}=-B^{*\omega}\circ P^{*\omega}.
\]
Using $B^{*\omega}=-B-\lambda\operatorname{id}_A$, we obtain
\[
-P\circ B-\lambda P
=
B\circ P^{*\omega}+\lambda P^{*\omega}.
\]
Together with $B\circ P^{*\omega}=-P\circ B$, this gives
\[
\lambda(P+P^{*\omega})=0.
\]
Consequently, if $\lambda\ne 0$, then
\[
P=-P^{*\omega}.
\]
\end{remark}

\begin{theorem}\label{thm:adjoint-averaging-descendent-Zinbiel}
Let $(A,\circ,B,\omega)$ be a quadratic Rota-Baxter Zinbiel algebra of
weight $\lambda$ with $\lambda\ne0$, and let $P:A\to A$ be an averaging
operator on $(A,\circ,B,\omega)$. Then $-P^{*\omega}$ is an averaging
operator on the descendent Zinbiel algebra $A_B=(A,\circ_B)$, where
\[
x\circ_B y
=
B(x)\circ y+x\circ B(y)+\lambda x\circ y,
\qquad \forall x,y\in A.
\]
\end{theorem}

\begin{proof}
Since $P$ is an averaging operator on $(A,\circ,B,\omega)$, we have
\[
B\circ P^{*\omega}=-P\circ B.
\]
By Remark~\ref{rem:averaging-operator-quadratic-rb-Zinbiel}, the condition
$\lambda\ne0$ implies
\[
P=-P^{*\omega}.
\]
Hence
\[
B\circ P=P\circ B.
\]

We now prove that $P$ is an averaging operator on the descendent Zinbiel
algebra $(A,\circ_B)$. For all $x,y\in A$, we have
\[
\begin{aligned}
P(x)\circ_B P(y)
&=
B(P(x))\circ P(y)+P(x)\circ B(P(y))+\lambda P(x)\circ P(y)\\
&=
P(B(x))\circ P(y)+P(x)\circ P(B(y))+\lambda P(x)\circ P(y).
\end{aligned}
\]
On the other hand,
\[
\begin{aligned}
P\bigl(P(x)\circ_B y\bigr)
&=
P\bigl(B(P(x))\circ y+P(x)\circ B(y)+\lambda P(x)\circ y\bigr)\\
&=
P\bigl(P(B(x))\circ y\bigr)
+P\bigl(P(x)\circ B(y)\bigr)
+\lambda P\bigl(P(x)\circ y\bigr)\\
&=
P(B(x))\circ P(y)+P(x)\circ P(B(y))+\lambda P(x)\circ P(y),
\end{aligned}
\]
where we used the averaging identity of $P$ on $(A,\circ)$. Thus
\[
P(x)\circ_B P(y)=P\bigl(P(x)\circ_B y\bigr).
\]
Similarly,
\[
\begin{aligned}
P\bigl(x\circ_B P(y)\bigr)
&=
P\bigl(B(x)\circ P(y)+x\circ B(P(y))+\lambda x\circ P(y)\bigr)\\
&=
P\bigl(B(x)\circ P(y)\bigr)
+P\bigl(x\circ P(B(y))\bigr)
+\lambda P\bigl(x\circ P(y)\bigr)\\
&=
P(B(x))\circ P(y)+P(x)\circ P(B(y))+\lambda P(x)\circ P(y).
\end{aligned}
\]
Therefore,
\[
P(x)\circ_B P(y)
=
P\bigl(P(x)\circ_B y\bigr)
=
P\bigl(x\circ_B P(y)\bigr),
\qquad \forall x,y\in A.
\]
Hence $P$ is an averaging operator on $(A,\circ_B)$.

Since $P=-P^{*\omega}$, it follows that $-P^{*\omega}$ is an averaging
operator on the descendent Zinbiel algebra $A_B$.
\end{proof}}

\section{Admissible Zinbiel Yang-Baxter equations and relative Rota-Baxter operators}\label{se:Admissible Zinbiel Yang-Baxter equations}
In this section, we construct averaging Zinbiel bialgebras from solutions
of a Zinbiel Yang-Baxter type equation and from relative Rota-Baxter
operators. We first formulate the \(S\) admissible Zinbiel Yang-Baxter
equation in an averaging Zinbiel algebra and prove that its solutions give
rise to averaging Zinbiel bialgebras. We then introduce relative Rota-Baxter operators on averaging Zinbiel algebras. Such operators produce
descendent averaging Zinbiel algebras, matched pairs, and solutions of the
\(S\) admissible Zinbiel Yang-Baxter equation.

\begin{lemma}\label{lem:admissible-Zinbiel-ybe}
Let $((A,\circ),P)$ be an $S$-admissible averaging Zinbiel algebra and let
$r\in A\otimes A.$
Define a linear map $\Delta_r:A\to A\otimes A$ by Eq.(\mref{eq:coboundary-Zinbiel}).
Suppose that $\Delta_r^*$ defines a Zinbiel algebra structure on $A^*$.
Then the following conclusions hold.

\begin{itemize}
\item[\emph{(a)}]
Eq.(\mref{eq:averaging Zinbiel coalgebra})
holds if and only if, for all $x\in A$,
\begin{equation}\label{eq:Zinbiel-lemma41-a1}
\begin{aligned}
&\bigl(S\otimes S\circ(L_x+R_x)-S\circ L_x\otimes S\bigr)(r)  \\
={}&
\bigl(S\otimes(L_{S(x)}+R_{S(x)})
      -L_{S(x)}\circ P\otimes\operatorname{id}\bigr)(r)\\
={}&
\bigl(\operatorname{id}\otimes (L_{S(x)}+R_{S(x)})\circ P
      -L_{S(x)}\otimes S\bigr)(r).
\end{aligned}
\end{equation}

\item[\emph{(b)}]
Eq.(\mref{eq:averaging Zinbiel bialgebraP*1})
holds if and only if, for all $x\in A$,
\begin{equation}\label{eq:Zinbiel-lemma41-b1}
\begin{aligned}
&\bigl(P\otimes S\circ(L_x+R_x)-P\circ L_x\otimes S\bigr)(r) \\
={}&
\bigl(\operatorname{id}\otimes (L_{P(x)}+R_{P(x)})\circ S
      -L_{P(x)}\otimes S\bigr)(r)\\
={}&
\bigl(P\otimes(L_{P(x)}+R_{P(x)})
      -L_{P(x)}\circ P\otimes\operatorname{id}\bigr)(r).
\end{aligned}
\end{equation}

\item[\emph{(c)}]
Eq.(\mref{eq:averaging Zinbiel bialgebraP*2})
holds if and only if, for all $x\in A$,
\begin{equation}\label{eq:Zinbiel-lemma41-c1}
\begin{aligned}
&\bigl(S\otimes P\circ(L_x+R_x)-S\circ L_x\otimes P\bigr)(r) \\
={}&
\bigl(S\otimes(L_{P(x)}+R_{P(x)})
      -L_{P(x)}\circ S\otimes\operatorname{id}\bigr)(r)\\
={}&
\bigl(\operatorname{id}\otimes (L_{P(x)}+R_{P(x)})\circ P
      -L_{P(x)}\otimes P\bigr)(r).
\end{aligned}
\end{equation}
\end{itemize}
\end{lemma}

\begin{proof}
We only prove \emph{(a)}, the proofs of (b) and (c) are similar.

For all $x\in A$, by Eq.\eqref{eq:coboundary-Zinbiel}, we have
\[
\begin{aligned}
(S\otimes S)\Delta_r(x)
&=
(S\otimes S)
\bigl(\operatorname{id}\otimes(L_x+R_x)-L_x\otimes\operatorname{id}\bigr)(r)\\
&=
\bigl(S\otimes S\circ(L_x+R_x)-S\circ L_x\otimes S\bigr)(r).
\end{aligned}
\]
On the other hand,
\[
\begin{aligned}
(S\otimes\operatorname{id})\Delta_r(S(x))
&=
(S\otimes\operatorname{id})
\bigl(\operatorname{id}\otimes(L_{S(x)}+R_{S(x)})
      -L_{S(x)}\otimes\operatorname{id}\bigr)(r)\\
&=
\bigl(S\otimes(L_{S(x)}+R_{S(x)})
      -S\circ L_{S(x)}\otimes\operatorname{id}\bigr)(r).
\end{aligned}
\]
Since $((A,\circ),P)$ is $S$-admissible, we have
\[
S(S(x)\circ y)=S(x)\circ P(y),
\qquad \forall x,y\in A.
\]
Hence
\[
S\circ L_{S(x)}=L_{S(x)}\circ P.
\]
Therefore,
\[
(S\otimes\operatorname{id})\Delta_r(S(x))
=
\bigl(S\otimes(L_{S(x)}+R_{S(x)})
      -L_{S(x)}\circ P\otimes\operatorname{id}\bigr)(r).
\]

Similarly,
\[
\begin{aligned}
(\operatorname{id}\otimes S)\Delta_r(S(x))
&=
(\operatorname{id}\otimes S)
\bigl(\operatorname{id}\otimes(L_{S(x)}+R_{S(x)})
      -L_{S(x)}\otimes\operatorname{id}\bigr)(r)\\
&=
\bigl(\operatorname{id}\otimes S\circ(L_{S(x)}+R_{S(x)})
      -L_{S(x)}\otimes S\bigr)(r).
\end{aligned}
\]
By the $S$ admissibility again,
\[
S\circ L_{S(x)}=L_{S(x)}\circ P,
\qquad
S\circ R_{S(x)}=R_{S(x)}\circ P.
\]
Thus
\[
S\circ(L_{S(x)}+R_{S(x)})
=
(L_{S(x)}+R_{S(x)})\circ P,
\]
and consequently
\[
(\operatorname{id}\otimes S)\Delta_r(S(x))
=
\bigl(\operatorname{id}\otimes(L_{S(x)}+R_{S(x)})\circ P
      -L_{S(x)}\otimes S\bigr)(r).
\]
This completes the proof.
\end{proof}

\begin{theorem}\label{thm:avg-Zinbiel-bialgebra-coboundary}
Let $((A,\circ),P)$ be an $S$-admissible averaging Zinbiel algebra and let
$r\in A\otimes A$.
Define a linear map $\Delta_r:A\to A\otimes A$ by Eq.(\mref{eq:coboundary-Zinbiel}).
Then $(A,\circ,\Delta_r,P,S)$ is an averaging Zinbiel bialgebra if and only
if Eqs.(\mref{eq:coboundary-Zinbiel-condition-one})-(\mref{eq:Zinbiel-lemma41-c1}) hold.
\end{theorem}

\begin{proof}
It is obvious.
\end{proof}

In particular, by further simplifying the above equivalent conditions, we
obtain the following theorem.

\begin{theorem}\label{thm:avg-Zinbiel-ybe-sufficient}
Let $((A,\circ),P)$ be an $S$-admissible averaging Zinbiel algebra and let
$r\in A\otimes A$.
Define $\Delta_r:A\to A\otimes A$ by Eq.(\mref{eq:coboundary-Zinbiel}).
Then $(A,\circ,\Delta_r,P,S)$ is an averaging Zinbiel bialgebra if
\eqref{eq:coboundary-Zinbiel-condition-one} and the following equations
hold:
\begin{equation}\label{eq:Zinbiel-ybe}
-r_{13}\circ r_{12}
-r_{23}\circ r_{21}
+r_{13}\cdot r_{21}
+r_{12}\cdot r_{23}
-r_{13}\cdot r_{23}
=0,
\end{equation}
\begin{equation}\label{eq:Zinbiel-S-id-id-P}
(S\otimes \operatorname{id}-\operatorname{id}\otimes P)(r)=0,
\end{equation}
and
\begin{equation}\label{eq:Zinbiel-P-id-id-S}
(P\otimes \operatorname{id}-\operatorname{id}\otimes S)(r)=0.
\end{equation}
\end{theorem}

\begin{proof}
It remains to prove that
\eqref{eq:averaging Zinbiel coalgebra},
\eqref{eq:averaging Zinbiel bialgebraP*1}, and
\eqref{eq:averaging Zinbiel bialgebraP*2} hold. We first prove
\eqref{eq:averaging Zinbiel coalgebra}. For all $x\in A$, we have
\[
\begin{aligned}
&(S\otimes S)\Delta_r(x)
-(S\otimes \operatorname{id})\Delta_r(S(x))\\
={}&
\bigl(S\otimes S\circ(L_x+R_x)-S\circ L_x\otimes S\bigr)(r)\\
&\quad
-\bigl(S\otimes (L_{S(x)}+R_{S(x)})
      -S\circ L_{S(x)}\otimes \operatorname{id}\bigr)(r).
\end{aligned}
\]
Using the $S$ admissibility of $((A,\circ),P)$, we have
\[
S\circ L_x\circ P=L_{S(x)}\circ P,
\qquad
S\circ R_x\circ P=R_{S(x)}\circ P,
\]
and
\[
S\circ L_{S(x)}=L_{S(x)}\circ P.
\]
Moreover, by \eqref{eq:Zinbiel-S-id-id-P} and
\eqref{eq:Zinbiel-P-id-id-S}, we get
\[
(S\otimes \operatorname{id})(r)=(\operatorname{id}\otimes P)(r),
\qquad
(P\otimes \operatorname{id})(r)=(\operatorname{id}\otimes S)(r).
\]
Hence
\[
\begin{aligned}
&(S\otimes S)\Delta_r(x)
-(S\otimes \operatorname{id})\Delta_r(S(x))\\
={}&
\bigl(\operatorname{id}\otimes S\circ(L_x+R_x)\circ P\bigr)(r)
-\bigl(S\circ L_x\circ P\otimes \operatorname{id}\bigr)(r)\\
&\quad
-\bigl(\operatorname{id}\otimes (L_{S(x)}+R_{S(x)})\circ P\bigr)(r)
+\bigl(L_{S(x)}\circ P\otimes \operatorname{id}\bigr)(r)\\
={}&0.
\end{aligned}
\]
Similarly,
\[
(S\otimes S)\Delta_r(x)
=
(\operatorname{id}\otimes S)\Delta_r(S(x)).
\]
Thus \eqref{eq:averaging Zinbiel coalgebra} holds.

Next, for all $x\in A$, a similar calculation gives
\[
(P\otimes S)\Delta_r(x)
=
(\operatorname{id}\otimes S)\Delta_r(P(x))
=
(P\otimes \operatorname{id})\Delta_r(P(x)).
\]
Indeed, this follows from
\[
S\circ L_x\circ S=L_{P(x)}\circ S,
\qquad
S\circ R_x\circ S=R_{P(x)}\circ S,
\]
and
\[
P\circ L_x\circ P=L_{P(x)}\circ P,
\]
which are consequences of the $S$ admissibility and the averaging identity
for $P$. Hence \eqref{eq:averaging Zinbiel bialgebraP*1} holds.

The proof of \eqref{eq:averaging Zinbiel bialgebraP*2} is the same. Therefore,
by Theorem~\ref{thm:avg-Zinbiel-bialgebra-coboundary},
$(A,\circ,\Delta_r,P,S)$ is an averaging Zinbiel bialgebra.
\end{proof}

\begin{definition}\label{def:S-admissible-Zinbiel-ybe}
Let $((A,\circ),P)$ be an averaging Zinbiel algebra. Suppose that
$r\in A\otimes A$ and $S:A\to A$ is a linear map. The Zinbiel Yang-Baxter
equation
\begin{equation}\label{eq:S-admissible-Zinbiel-ybe-main}
\llbracket r,r\rrbracket
=
-r_{13}\circ r_{12}
-r_{23}\circ r_{21}
+r_{13}\cdot r_{21}
+r_{12}\cdot r_{23}
-r_{13}\cdot r_{23}
=0
\end{equation}
together with
\begin{equation}\label{eq:S-admissible-Zinbiel-ybe-P-S}
(S\otimes \operatorname{id}-\operatorname{id}\otimes P)(r)=0
\end{equation}
and
\begin{equation}\label{eq:S-admissible-Zinbiel-ybe-S-P}
(P\otimes \operatorname{id}-\operatorname{id}\otimes S)(r)=0
\end{equation}
is called the {\bf $S$ averaging Zinbiel Yang-Baxter equation}, or the
{\bf $S$-admissible Zinbiel Yang-Baxter equation} in the averaging Zinbiel algebra
$((A,\circ),P)$.
\end{definition}

\begin{remark}\label{rem:symmetric-S-admissible-Zinbiel-ybe}
Let $r=\sum_{i=1}^n a_i\otimes b_i.$ If $r$ is symmetric, then
\[
\sum_{i=1}^n S(a_i)\otimes b_i
=
\sum_{i=1}^n a_i\otimes P(b_i)
\]
is equivalent to
\[
\sum_{i=1}^n P(a_i)\otimes b_i
=
\sum_{i=1}^n a_i\otimes S(b_i).
\]
Therefore,
\[
(S\otimes \operatorname{id})(r)=(\operatorname{id}\otimes P)(r)
\]
holds if and only if
\[
(P\otimes \operatorname{id})(r)=(\operatorname{id}\otimes S)(r)
\]
holds.
\end{remark}

By Theorem~\ref{thm:avg-Zinbiel-ybe-sufficient} and
Remark~\ref{rem:symmetric-S-admissible-Zinbiel-ybe}, we obtain the following
result.

\begin{theorem}\label{thm:symmetric-S-admissible-Zinbiel-ybe}
Let $((A,\circ),P)$ be an $S$-admissible averaging Zinbiel algebra, and let
$\Delta_r:A\to A\otimes A$ be given by
\[
\Delta_r(x)
=
\bigl(\operatorname{id}\otimes (L_x+R_x)-L_x\otimes \operatorname{id}\bigr)r,
\qquad \forall x\in A.
\]
If $r\in A\otimes A$ is a symmetric solution of the $S$-admissible Zinbiel
Yang-Baxter equation in $((A,\circ),P)$, then
\[
(A,\circ,\Delta_r,P,S)
\]
is an averaging Zinbiel bialgebra.
\end{theorem}

\begin{proof}
Since $r$ is symmetric, its skew symmetric part is zero. Hence the condition
\[
\bigl(
L_{x\circ y}\otimes \operatorname{id}
-\operatorname{id}\otimes L_{x\circ y}
-L_xL_y\otimes \operatorname{id}
+L_x\otimes L_y
\bigr)(a)=0
\]
holds automatically.

Moreover, by Remark~\ref{rem:symmetric-S-admissible-Zinbiel-ybe} and
\eqref{eq:S-admissible-Zinbiel-ybe-P-S}, we also have
\[
(P\otimes \operatorname{id})(r)=(\operatorname{id}\otimes S)(r).
\]
Thus the assumptions of
Theorem~\ref{thm:avg-Zinbiel-ybe-sufficient} are satisfied. Therefore,
\[
(A,\circ,\Delta_r,P,S)
\]
is an averaging Zinbiel bialgebra.
\end{proof}

Next, we introduce the notion of a relative Rota-Baxter operator on an averaging
Zinbiel algebra.

\begin{definition}\label{def:relative-rb-avg-Zinbiel}
Let $((V;\rho,\mu),\alpha)$ be a representation of an averaging Zinbiel
algebra $((A,\circ),P)$. A linear map $T:V\to A$ is called a {\bf relative
Rota-Baxter operator} on $((A,\circ),P)$ with respect to the representation
$((V;\rho,\mu),\alpha)$ if the following equations hold:
\begin{equation}\label{eq:relative-rb-Zinbiel}
T(u)\circ T(v)
=
T\bigl(\rho(T(u))v+\mu(T(v))u\bigr),
\qquad \forall u,v\in V,
\end{equation}
and
\begin{equation}\label{eq:relative-rb-Zinbiel-avg}
P\circ T=T\circ \alpha.
\end{equation}
Moreover, a relative Rota-Baxter operator on an averaging Zinbiel algebra
with respect to the regular representation $((A;L,R),P)$ is called a
Rota-Baxter operator of weight $0$.
\end{definition}

\begin{proposition}\label{prop:relative-rb-induced-avg-Zinbiel}
Let $T:V\to A$ be a relative Rota-Baxter operator on an averaging Zinbiel
algebra $((A,\circ),P)$ with respect to the representation
$((V;\rho,\mu),\alpha)$. Define a bilinear product $\circ_T$ on $V$ by
\begin{equation}\label{eq:descendent-product-relative-rb-Zinbiel}
u\circ_T v
:=
\rho(T(u))v+\mu(T(v))u,
\qquad \forall u,v\in V.
\end{equation}
Then $((V,\circ_T),\alpha)$ is an averaging Zinbiel algebra, which is called
the descendent averaging Zinbiel algebra.
\end{proposition}

\begin{proof}
For all $u,v\in V$, by \eqref{eq:relative-rb-Zinbiel}, we have
\[
T(u\circ_T v)=T(u)\circ T(v).
\]
We first prove that $(V,\circ_T)$ is a Zinbiel algebra. For all
$u,v,w\in V$, we have
\[
\begin{aligned}
u\circ_T(v\circ_T w)
={}&
\rho(Tu)\bigl(\rho(Tv)w+\mu(Tw)v\bigr)
+\mu\bigl(T(v\circ_T w)\bigr)u                                      \\
={}&
\rho(Tu)\rho(Tv)w+\rho(Tu)\mu(Tw)v
+\mu(Tv\circ Tw)u.
\end{aligned}
\]
Using the representation identities of a Zinbiel algebra,
\[
\rho(x)\rho(y)=\rho(x\circ y)+\rho(y\circ x),
\qquad
\rho(x)\mu(y)=\mu(x\circ y),
\]
we get
\[
\begin{aligned}
u\circ_T(v\circ_T w)
={}&
\rho(Tu\circ Tv)w+\rho(Tv\circ Tu)w
+\mu(Tu\circ Tw)v+\mu(Tv\circ Tw)u.
\end{aligned}
\]
On the other hand,
\[
\begin{aligned}
&(u\circ_T v+v\circ_T u)\circ_T w                                      \\
={}&
\rho\bigl(T(u\circ_T v+v\circ_T u)\bigr)w
+\mu(Tw)(u\circ_T v+v\circ_T u)                                      \\
={}&
\rho(Tu\circ Tv+Tv\circ Tu)w                                      \\
&\quad
+\mu(Tw)\bigl(\rho(Tu)v+\mu(Tv)u+\rho(Tv)u+\mu(Tu)v\bigr).
\end{aligned}
\]
Since
\[
\mu(y)\rho(x)+\mu(y)\mu(x)=\mu(x\circ y),
\]
we obtain
\[
\begin{aligned}
(u\circ_T v+v\circ_T u)\circ_T w
={}&
\rho(Tu\circ Tv)w+\rho(Tv\circ Tu)w  \\
&\quad
+\mu(Tu\circ Tw)v+\mu(Tv\circ Tw)u.
\end{aligned}
\]
Thus
\[
u\circ_T(v\circ_T w)
=
(u\circ_T v+v\circ_T u)\circ_T w,
\qquad \forall u,v,w\in V.
\]
Hence $(V,\circ_T)$ is a Zinbiel algebra.

It remains to prove that $\alpha$ is an averaging operator on
$(V,\circ_T)$. For all $u,v\in V$, by \eqref{eq:relative-rb-Zinbiel-avg},
we have $T(\alpha(u))=P(T(u))$ and $T(\alpha(v))=P(T(v))$. Therefore,
\[
\begin{aligned}
\alpha(u)\circ_T\alpha(v)
&=
\rho(T(\alpha(u)))\alpha(v)+\mu(T(\alpha(v)))\alpha(u)\\
&=
\rho(P(Tu))\alpha(v)+\mu(P(Tv))\alpha(u).
\end{aligned}
\]
Using that $((V;\rho,\mu),\alpha)$ is a representation of the averaging
Zinbiel algebra $((A,\circ),P)$, we have
\[
\rho(Px)\alpha(v)
=
\alpha(\rho(Px)v)
=
\alpha(\rho(x)\alpha(v)),
\]
and
\[
\mu(Px)\alpha(v)
=
\alpha(\mu(Px)v)
=
\alpha(\mu(x)\alpha(v)).
\]
Hence
\[
\begin{aligned}
\alpha(u)\circ_T\alpha(v)
&=
\alpha\bigl(\rho(P(Tu))v\bigr)
+\alpha\bigl(\mu(P(Tv))u\bigr)                                      \\
&=
\alpha\bigl(\rho(T(\alpha(u)))v+\mu(Tv)\alpha(u)\bigr)               \\
&=
\alpha\bigl(\alpha(u)\circ_T v\bigr).
\end{aligned}
\]
Similarly,
\[
\begin{aligned}
\alpha(u)\circ_T\alpha(v)
&=
\alpha\bigl(\rho(Tu)\alpha(v)+\mu(T(\alpha(v)))u\bigr)\\
&=
\alpha\bigl(u\circ_T\alpha(v)\bigr).
\end{aligned}
\]
Therefore,
\[
\alpha(u)\circ_T\alpha(v)
=
\alpha\bigl(\alpha(u)\circ_T v\bigr)
=
\alpha\bigl(u\circ_T\alpha(v)\bigr),
\qquad \forall u,v\in V.
\]
Thus $\alpha$ is an averaging operator on $(V,\circ_T)$. Consequently,
$((V,\circ_T),\alpha)$ is an averaging Zinbiel algebra.
\end{proof}

\begin{corollary}\label{cor:relative-rb-homomorphism-avg-Zinbiel}
Let $T:V\to A$ be a relative Rota-Baxter operator on an averaging Zinbiel
algebra $((A,\circ),P)$ with respect to the representation
$((V;\rho,\mu),\alpha)$. Then $T$ is a homomorphism from the descendent
averaging Zinbiel algebra $((V,\circ_T),\alpha)$ to the averaging Zinbiel
algebra $((A,\circ),P)$.
\end{corollary}

\begin{proposition}\label{prop:relative-rb-matched-pair-avg-Zinbiel}
Let $T:V\to A$ be a relative Rota-Baxter operator on an averaging Zinbiel
algebra $((A,\circ),P)$ with respect to the representation
$((V;\rho,\mu),\alpha)$. Then
\[
\Bigl(
((A,\circ),P),\
((V,\circ_T),\alpha),\
\rho,\mu,\rho_T,\mu_T
\Bigr)
\]
is a matched pair of averaging Zinbiel algebras, where
$((V,\circ_T),\alpha)$ is the descendent averaging Zinbiel algebra with
\[
u\circ_T v
=
\rho(T(u))v+\mu(T(v))u,
\qquad \forall u,v\in V,
\]
and $\rho_T,\mu_T:V\to \operatorname{End}(A)$ are defined by
\[
\rho_T(u)x
=
T(u)\circ x-T(\mu(x)u),
\]
and
\[
\mu_T(u)x
=
x\circ T(u)-T(\rho(x)u),
\]
for all $u\in V$ and $x\in A$.
\end{proposition}

\begin{proof}
By Proposition~\ref{prop:relative-rb-induced-avg-Zinbiel},
$((V,\circ_T),\alpha)$ is an averaging Zinbiel algebra. Since
$((V;\rho,\mu),\alpha)$ is a representation of $((A,\circ),P)$ by assumption,
we only need to check that $((A;\rho_T,\mu_T),P)$ is a representation of
$((V,\circ_T),\alpha)$ as an averaging Zinbiel algebra, and then use the
ordinary matched pair result induced by a relative Rota-Baxter operator.

First, for all $u\in V$ and $x\in A$, we have
\[
\begin{aligned}
\rho_T(\alpha(u))P(x)
&=
T(\alpha(u))\circ P(x)-T(\mu(P(x))\alpha(u))  \\
&=
P(T(u))\circ P(x)-T(\mu(P(x))\alpha(u)).
\end{aligned}
\]
Using the averaging identity of $P$ and the representation compatibility of
$((V;\rho,\mu),\alpha)$, we obtain
\[
\begin{aligned}
\rho_T(\alpha(u))P(x)
&=
P(P(T(u))\circ x)-T(\alpha(\mu(x)\alpha(u)))\\
&=
P(P(T(u))\circ x)-P(T(\mu(x)\alpha(u)))\\
&=
P\bigl(T(\alpha(u))\circ x-T(\mu(x)\alpha(u))\bigr)\\
&=
P\bigl(\rho_T(\alpha(u))x\bigr).
\end{aligned}
\]
On the other hand,
\[
\begin{aligned}
P\bigl(\rho_T(u)P(x)\bigr)
&=
P\bigl(T(u)\circ P(x)-T(\mu(P(x))u)\bigr)\\
&=
P(T(u)\circ P(x))-T(\alpha(\mu(P(x))u))\\
&=
P(T(u)\circ P(x))-T(\mu(P(x))\alpha(u))\\
&=
P(T(u))\circ P(x)-T(\mu(P(x))\alpha(u))\\
&=
\rho_T(\alpha(u))P(x).
\end{aligned}
\]
Hence
\[
\rho_T(\alpha(u))P(x)
=
P\bigl(\rho_T(\alpha(u))x\bigr)
=
P\bigl(\rho_T(u)P(x)\bigr).
\]

Similarly, for $\mu_T$, we have
\[
\begin{aligned}
\mu_T(\alpha(u))P(x)
&=
P(x)\circ T(\alpha(u))-T(\rho(P(x))\alpha(u))\\
&=
P(x)\circ P(T(u))-T(\rho(P(x))\alpha(u)).
\end{aligned}
\]
Then
\[
\begin{aligned}
P\bigl(\mu_T(\alpha(u))x\bigr)
&=
P\bigl(x\circ T(\alpha(u))-T(\rho(x)\alpha(u))\bigr)\\
&=
P(x\circ P(T(u)))-T(\alpha(\rho(x)\alpha(u)))\\
&=
P(x)\circ P(T(u))-T(\rho(P(x))\alpha(u))\\
&=
\mu_T(\alpha(u))P(x),
\end{aligned}
\]
and
\[
\begin{aligned}
P\bigl(\mu_T(u)P(x)\bigr)
&=
P\bigl(P(x)\circ T(u)-T(\rho(P(x))u)\bigr)\\
&=
P(P(x)\circ T(u))-T(\alpha(\rho(P(x))u))\\
&=
P(x)\circ P(T(u))-T(\rho(P(x))\alpha(u))\\
&=
\mu_T(\alpha(u))P(x).
\end{aligned}
\]
Therefore,
\[
\mu_T(\alpha(u))P(x)
=
P\bigl(\mu_T(\alpha(u))x\bigr)
=
P\bigl(\mu_T(u)P(x)\bigr).
\]
Thus $((A;\rho_T,\mu_T),P)$ is compatible with the averaging operator.

Moreover, in the ordinary Zinbiel setting, a relative Rota-Baxter operator
$T:V\to A$ gives a matched pair of Zinbiel algebras
\[
(A,V;(\rho,\mu),(\rho_T,\mu_T)),
\]
where $V$ is endowed with the descendent Zinbiel product $\circ_T$. Hence
$(A;\rho_T,\mu_T)$ is a representation of $(V,\circ_T)$, and
\[
(A,V;(\rho,\mu),(\rho_T,\mu_T))
\]
is a matched pair of Zinbiel algebras.

Combining this ordinary matched pair with the above compatibility conditions
for the averaging operators, we conclude that
\[
\Bigl(
((A,\circ),P),\
((V,\circ_T),\alpha),\
\rho,\mu,\rho_T,\mu_T
\Bigr)
\]
is a matched pair of averaging Zinbiel algebras.
\end{proof}

\begin{proposition}\label{prop:relative-rb-coregular-gives-S-admissible-Zinbiel-ybe}
Let $((A,\circ),P)$ be an averaging Zinbiel algebra. Suppose that
$r\in A\otimes A$ is symmetric and that $S:A\to A$ is a linear map.
Let
\[
r^\sharp:A^*\to A
\]
be defined by
\[
r^\sharp(\xi)=\sum_{i=1}^n \langle \xi,a_i\rangle b_i,
\qquad
\text{for } r=\sum_{i=1}^n a_i\otimes b_i.
\]
If $r^\sharp:A^*\to A$ is a relative Rota-Baxter operator on
$((A,\circ),P)$ with respect to the representation
\[
\bigl((A^*;-L^*-R^*,R^*),S^*\bigr),
\]
then $r$ is a solution of the $S$-admissible Zinbiel Yang-Baxter equation
in the averaging Zinbiel algebra $((A,\circ),P)$.
\end{proposition}

\begin{proof}
Since $r^\sharp:A^*\to A$ is a relative Rota-Baxter operator on
$((A,\circ),P)$ with respect to
$\bigl((A^*;-L^*-R^*,R^*),S^*\bigr)$, we have
\begin{equation}\label{eq:coregular-relative-rb-Zinbiel}
r^\sharp(\xi)\circ r^\sharp(\eta)
=
r^\sharp\Bigl(
(-L^*_{r^\sharp(\xi)}-R^*_{r^\sharp(\xi)})\eta
+
R^*_{r^\sharp(\eta)}\xi
\Bigr),
\qquad \forall \xi,\eta\in A^*,
\end{equation}
and
\begin{equation}\label{eq:rsharp-avg-compatible}
P\circ r^\sharp=r^\sharp\circ S^*.
\end{equation}
By the ordinary Zinbiel algebra result relating symmetric solutions of the
Zinbiel Yang-Baxter equation and relative Rota-Baxter operators with respect
to the coregular representation, \eqref{eq:coregular-relative-rb-Zinbiel}
implies that $r$ is a solution of the Zinbiel Yang-Baxter equation, namely
\[
\llbracket r,r\rrbracket=0.
\]

It remains to verify the averaging admissibility conditions. For any
$\xi\in A^*$, by \eqref{eq:rsharp-avg-compatible}, we have
\[
P r^\sharp(\xi)=r^\sharp(S^*\xi).
\]
Since $r=\sum_{i=1}^n a_i\otimes b_i$, this gives
\[
P\left(\sum_{i=1}^n \langle \xi,a_i\rangle b_i\right)
=
\sum_{i=1}^n \langle S^*\xi,a_i\rangle b_i.
\]
Hence
\[
\sum_{i=1}^n \langle \xi,a_i\rangle P(b_i)
=
\sum_{i=1}^n \langle \xi,S(a_i)\rangle b_i,
\qquad \forall \xi\in A^*.
\]
Therefore,
\[
(\operatorname{id}\otimes P)(r)=(S\otimes \operatorname{id})(r),
\]
that is,
\[
(S\otimes \operatorname{id}-\operatorname{id}\otimes P)(r)=0.
\]
Since $r$ is symmetric, this condition is equivalent to
\[
(P\otimes \operatorname{id}-\operatorname{id}\otimes S)(r)=0.
\]
Thus $r$ satisfies
\[
\llbracket r,r\rrbracket=0,
\qquad
(S\otimes \operatorname{id}-\operatorname{id}\otimes P)(r)=0,
\qquad
(P\otimes \operatorname{id}-\operatorname{id}\otimes S)(r)=0.
\]
Hence $r$ is a solution of the $S$-admissible Zinbiel Yang-Baxter equation
in $((A,\circ),P)$. This completes the proof.
\end{proof}

\begin{proposition}\label{prop:relative-rb-coregular-vs-rb-Zinbiel}
Let $((A,\circ),P,\omega)$ be a quadratic averaging Zinbiel algebra and let
$T:A^*\to A$ be a linear map. Then $T$ is a relative Rota-Baxter operator on
$((A,\circ),P)$ with respect to the coregular representation
\[
\bigl((A^*;-L^*-R^*,R^*),P^*\bigr)
\]
if and only if
\[
B:=T\circ\omega^\sharp:A\to A
\]
is a Rota-Baxter operator of weight $0$ on the averaging Zinbiel algebra
$((A,\circ),P)$.
\end{proposition}

\begin{proof}
For all $x,y\in A$, by
Proposition~\ref{pro:quadratic-avg-Zinbiel-isomorphism}, we have
\[
(-L^*_{B(x)}-R^*_{B(x)})\omega^\sharp(y)
=
\omega^\sharp\bigl(B(x)\circ y\bigr),
\]
and
\[
R^*_{B(y)}\omega^\sharp(x)
=
\omega^\sharp\bigl(x\circ B(y)\bigr).
\]
Therefore,
\[
\begin{aligned}
&T\Bigl(
(-L^*_{B(x)}-R^*_{B(x)})\omega^\sharp(y)
+
R^*_{B(y)}\omega^\sharp(x)
\Bigr)\\
={}&
T\Bigl(
\omega^\sharp(B(x)\circ y)
+
\omega^\sharp(x\circ B(y))
\Bigr)\\
={}&
T\omega^\sharp\bigl(B(x)\circ y+x\circ B(y)\bigr)\\
={}&
B\bigl(B(x)\circ y+x\circ B(y)\bigr).
\end{aligned}
\]
Since $B(x)=T\omega^\sharp(x)$ and $B(y)=T\omega^\sharp(y)$, the relative
Rota-Baxter identity for $T$,
\[
T(\xi)\circ T(\eta)
=
T\Bigl(
(-L^*_{T(\xi)}-R^*_{T(\xi)})\eta
+
R^*_{T(\eta)}\xi
\Bigr),
\qquad \forall \xi,\eta\in A^*,
\]
is equivalent to
\[
B(x)\circ B(y)
=
B\bigl(B(x)\circ y+x\circ B(y)\bigr),
\qquad \forall x,y\in A.
\]
This means that $B=T\circ\omega^\sharp$ is a Rota-Baxter operator of
weight $0$ on the Zinbiel algebra $(A,\circ)$.

It remains to check the compatibility with the averaging operators. Since
$\omega^\sharp$ is an isomorphism from the regular representation to the
coregular representation, we have
\[
P^*\circ\omega^\sharp=\omega^\sharp\circ P.
\]
Hence
\[
P\circ T=T\circ P^*
\]
if and only if
\[
P\circ T\circ\omega^\sharp
=
T\circ P^*\circ\omega^\sharp
=
T\circ\omega^\sharp\circ P,
\]
that is,
\[
P\circ B=B\circ P.
\]
Therefore, $T$ is compatible with $P$ and $P^*$ if and only if
$B=T\circ\omega^\sharp$ is compatible with $P$.

Combining the product identity and the averaging compatibility, we obtain
that $T$ is a relative Rota-Baxter operator on $((A,\circ),P)$ with respect
to $\bigl((A^*;-L^*-R^*,R^*),P^*\bigr)$ if and only if
$T\circ\omega^\sharp$ is a Rota-Baxter operator of weight $0$ on
$((A,\circ),P)$. The proof is complete.
\end{proof}

\begin{corollary}\label{cor:symmetric-r-from-rb-avg-Zinbiel}
Let $((A,\circ),P,\omega)$ be a quadratic averaging Zinbiel algebra.
Suppose that $r\in A\otimes A$ is symmetric. If $r^\sharp\circ \omega^\sharp:A\to A$ is a Rota-Baxter operator of weight $0$ on the averaging Zinbiel algebra
$((A,\circ),P)$, then $r$ is a solution of the $P$ admissible Zinbiel
Yang-Baxter equation in $((A,\circ),P)$.
\end{corollary}

\begin{proof}
By Proposition~\ref{prop:relative-rb-coregular-vs-rb-Zinbiel}, the condition
that $r^\sharp\circ \omega^\sharp$ is a Rota-Baxter operator of weight $0$
on $((A,\circ),P)$ is equivalent to saying that $r^\sharp$ is a relative Rota-Baxter operator on $((A,\circ),P)$ with respect to the
coregular representation $\bigl((A^*;-L^*-R^*,R^*),P^*\bigr).$
Since $r$ is symmetric, Proposition~\ref{prop:relative-rb-coregular-gives-S-admissible-Zinbiel-ybe}
with $S=P$ implies that $r$ is a solution of the $P$ admissible Zinbiel
Yang-Baxter equation in $((A,\circ),P)$.
\end{proof}

\begin{lemma}\label{lem:relative-rb-Zinbiel-ybe-semi-direct}
Let $(A,\circ)$ be a Zinbiel algebra and let $(V;\rho,\mu)$ be a
representation of $(A,\circ)$. Then
\[
\mathcal A:=A\ltimes_{-\rho^*-\mu^*,\,\mu^*}V^*
\]
is a Zinbiel algebra, where
\[
(x+\xi)\circ_{\mathcal A}(y+\eta)
=
x\circ y+(-\rho^*(x)-\mu^*(x))\eta+\mu^*(y)\xi,
\]
for all $x,y\in A$ and $\xi,\eta\in V^*$.

Let $T:V\to A$ be a linear map. For a basis $\{v_1,\ldots,v_n\}$ of $V$
and its dual basis $\{v_1^*,\ldots,v_n^*\}$, define
\[
\overline T:=\sum_{i=1}^n T(v_i)\otimes v_i^*
\in \mathcal A\otimes \mathcal A,
\]
and
\[
r_T:=\overline T+\sigma(\overline T)
=
\sum_{i=1}^n
\bigl(T(v_i)\otimes v_i^*+v_i^*\otimes T(v_i)\bigr).
\]
Then $r_T$ is a symmetric solution of the Zinbiel Yang-Baxter equation in
$\mathcal A$ if and only if $T$ is a relative Rota-Baxter operator on
$(A,\circ)$ with respect to the representation $(V;\rho,\mu)$, namely
\[
T(u)\circ T(v)
=
T\bigl(\rho(T(u))v+\mu(T(v))u\bigr),
\qquad \forall u,v\in V.
\]
\end{lemma}

\begin{proof}
The tensor $r_T$ is symmetric by construction. By expanding the Zinbiel
Yang-Baxter equation in the semi-direct product
$\mathcal A=A\ltimes_{-\rho^*-\mu^*,\,\mu^*}V^*$ and pairing it with
elements of $\mathcal A^*=A^*\oplus V$, all nonzero terms are exactly
equivalent to
\[
T(u)\circ T(v)
=
T\bigl(\rho(T(u))v+\mu(T(v))u\bigr),
\qquad \forall u,v\in V.
\]
Thus $r_T$ is a solution of the Zinbiel Yang-Baxter equation in
$\mathcal A$ if and only if $T$ is a relative Rota-Baxter operator.
\end{proof}

\begin{theorem}\label{thm:relative-rb-gives-admissible-Zinbiel-ybe}
Let $T:V\to A$ be a relative Rota-Baxter operator on an averaging Zinbiel
algebra $((A,\circ),P)$ with respect to the representation
$((V;\rho,\mu),\alpha)$. Then
\[
r_T=\overline T+\sigma(\overline T)
\]
is a symmetric solution of the admissible Zinbiel Yang-Baxter equation in
the semi-direct product averaging Zinbiel algebra
\[
\left(
A\ltimes_{-\rho^*-\mu^*,\,\mu^*}V^*,\ P\oplus \alpha^*
\right),
\]
where
\[
\overline T=\sum_{i=1}^n T(v_i)\otimes v_i^*
\]
for any basis $\{v_1,\ldots,v_n\}$ of $V$ and its dual basis
$\{v_1^*,\ldots,v_n^*\}$.
\end{theorem}

\begin{proof}
Set
\[
\mathcal A:=A\ltimes_{-\rho^*-\mu^*,\,\mu^*}V^*,
\qquad
\mathcal P:=P\oplus \alpha^*.
\]
Since $((V;\rho,\mu),\alpha)$ is a representation of $((A,\circ),P)$, the
dual representation
\[
\bigl((V^*;-\rho^*-\mu^*,\mu^*),\alpha^*\bigr)
\]
is also a representation of $((A,\circ),P)$. Therefore,
$(\mathcal A,\mathcal P)$ is an averaging Zinbiel algebra.

By Lemma~\ref{lem:relative-rb-Zinbiel-ybe-semi-direct}, the tensor
$r_T=\overline T+\sigma(\overline T)$ is a symmetric solution of the
Zinbiel Yang-Baxter equation in $\mathcal A$. It remains to verify the
admissibility condition $(\mathcal P\otimes \operatorname{id}
-\operatorname{id}\otimes \mathcal P)r_T=0.$
We first prove this condition for $\overline T$. For
$\xi,\eta\in A^*$ and $u,v\in V$, regard $\overline T$ as a bilinear form
on $\mathcal A^*=A^*\oplus V$ by $\overline T(\xi+u,\eta+v)=\langle T(u),\eta\rangle.$
Then
\[
\begin{aligned}
&\bigl((\mathcal P\otimes \operatorname{id}
-\operatorname{id}\otimes \mathcal P)\overline T\bigr)
(\xi+u,\eta+v)\\
={}&
\overline T(\mathcal P^*(\xi+u),\eta+v)
-\overline T(\xi+u,\mathcal P^*(\eta+v))\\
={}&
\overline T(P^*\xi+\alpha(u),\eta+v)
-\overline T(\xi+u,P^*\eta+\alpha(v))\\
={}&
\langle T(\alpha(u)),\eta\rangle-\langle T(u),P^*\eta\rangle\\
={}&
\langle T(\alpha(u))-P(T(u)),\eta\rangle\\
={}&0,
\end{aligned}
\]
where we used the compatibility condition $P\circ T=T\circ \alpha.$ Hence $(\mathcal P\otimes \operatorname{id}
-\operatorname{id}\otimes \mathcal P)\overline T=0.$

Similarly,
\[
\begin{aligned}
&\bigl((\mathcal P\otimes \operatorname{id}
-\operatorname{id}\otimes \mathcal P)\sigma(\overline T)\bigr)
(\xi+u,\eta+v)\\
={}&
\sigma(\overline T)(P^*\xi+\alpha(u),\eta+v)
-\sigma(\overline T)(\xi+u,P^*\eta+\alpha(v))\\
={}&
\langle T(v),P^*\xi\rangle-\langle T(\alpha(v)),\xi\rangle\\
={}&
\langle P(T(v))-T(\alpha(v)),\xi\rangle\\
={}&0.
\end{aligned}
\]
Therefore,
\[
(\mathcal P\otimes \operatorname{id}
-\operatorname{id}\otimes \mathcal P)\sigma(\overline T)=0.
\]
Consequently,
\[
(\mathcal P\otimes \operatorname{id}
-\operatorname{id}\otimes \mathcal P)r_T=0.
\]

Since $r_T$ is symmetric, this condition is exactly the admissibility
condition for the admissible Zinbiel Yang-Baxter equation with respect to
the averaging operator $\mathcal P$. Thus $r_T$ is a symmetric solution of
the admissible Zinbiel Yang-Baxter equation in
\[
\left(
A\ltimes_{-\rho^*-\mu^*,\,\mu^*}V^*,\ P\oplus \alpha^*
\right).
\]
\end{proof}

\section*{Acknowledgment}
The work is supported by Natural Science Foundation of China (No.12271292, 12471036).


\begin{thebibliography}{99}

\bibitem{Aguiar2000}
M. Aguiar, Pre-Poisson algebras, Lett. Math. Phys. 54 (2000), 263--277.

\bibitem{Aguiar2001}
M. Aguiar, On the associative analog of Lie bialgebras, J. Algebra 244
(2001), 492--532.

\bibitem{Aguiar2004}
M. Aguiar, Infinitesimal bialgebras, pre-Lie and dendriform algebras, in:
Hopf Algebras, Lecture Notes in Pure and Appl. Math. 237, Marcel Dekker,
New York, 2004, pp. 1--33.


\bibitem{Bai2008}
C. Bai, Left symmetric bialgebras and an analogue of the classical
Yang-Baxter equation, Commun. Contemp. Math. 10 (2008), 221--260.


\bibitem{Cao2000}
W. Cao, An algebraic study of averaging operators, PhD thesis, Rutgers
University at Newark, 2000.

\bibitem{Chapoton2022}
F. Chapoton, Zinbiel algebras and multiple zeta values, Doc. Math. 27
(2022), 519--533.


\bibitem{CovezFarinatiLebedManchon2023}
S. Covez, M. Farinati, V. Lebed and D. Manchon, Bialgebraic approach to
rack cohomology, Algebr. Geom. Topol. 23 (2023), 1551--1582.

\bibitem{Drinfeld1983}
V. G. Drinfel'd, Hamiltonian structures on Lie groups, Lie bialgebras and
the geometric meaning of classical Yang-Baxter equations, Sov. Math. Dokl.
27 (1983), 68--71.


\bibitem{GaoGuoZhang2021}
X. Gao, L. Guo and Y. Zhang, Cocommutative matching Rota-Baxter operators,
shuffle products with decorations and matching Zinbiel algebras, J. Algebra
586 (2021), 402--432.

\bibitem{GaoYangZhang2026}
L. Gao, M. Yang and Y. Zhang, Averaging pre-Lie bialgebras,
arXiv:2602.10671.

\bibitem{GinzburgKapranov1994}
V. Ginzburg and M. Kapranov, Koszul duality for operads, Duke Math. J. 76
(1994), 203--272.


\bibitem{Guo2012}
L. Guo, An Introduction to Rota-Baxter Algebra, Surveys of Modern
Mathematics, Vol. 4, International Press, Somerville, MA, Higher Education
Press, Beijing, 2012.

\bibitem{HouCui2027}Z. Cui and B. Hou, Averaging antisymmetric infinitesimal bialgebras and induced perm bialgebras, J. Algebra Appl. (2027), 2750126, 34pp.

\bibitem{HouShengZhou2026}
S. Hou, Y. Sheng and Y. Zhou, Averaging Lie bialgebras, J. Algebra 689
(2026), 390--422.

\bibitem{KosmannSchwarzbach2004}
Y. Kosmann Schwarzbach, Lie bialgebras, Poisson Lie groups and dressing
transformations, in: Integrability of Nonlinear Systems, Lecture Notes in
Physics 638, Springer, Berlin, 2004, pp. 107--173.


\bibitem{LodayPirashvili1993}
J.-L. Loday and T. Pirashvili, Universal enveloping algebras of Leibniz
algebras and (co)homology, Math. Ann. 296 (1993), 139--158.

\bibitem{LodayVallette2012}
J.-L. Loday and B. Vallette, Algebraic Operads, Springer, Berlin, 2012.

\bibitem{Reynolds1895}
O. Reynolds, On the dynamical theory of incompressible viscous fluids and
the determination of the criterion, Philos. Trans. R. Soc. Lond. A 186
(1895), 123--164.


\bibitem{ShengTangZhu2023}
Y. Sheng, R. Tang and C. Zhu, The controlling $L_\infty$-algebra, cohomology and homotopy of embedding tensors and Lie-Leibniz triples, Commun. Math. Phys. 386 (2021), 269-304.

\bibitem{WY2026}
Y. Wang, Zinbiel bialgebras, relative Rota-Baxter operators and the related
Yang-Baxter equation, J. Algebra 689 (2026) 656-689.

\bibitem{ZhaoLiu} Q. Zhao, G. Liu, A further study on averaging commutative and cocommutative infinitesimal bialgebras and special apre-perm bialgebras, arXiv:2510.09208
\end{thebibliography}
 \end{document}